\documentclass{amsart}
\title{Squeezed knots}
\author{Peter Feller}
\address{ETH Zurich, R\"amistrasse 101, 8092 Zurich, Switzerland}
\email{\myemail{peter.feller@math.ch}}
\urladdr{\url{https://people.math.ethz.ch/~pfeller/}}
\author{Lukas Lewark}
\address{Faculty of Mathematics, University of Regensburg, 93053 Regensburg, Germany}
\email{\myemail{lukas@lewark.de}}
\urladdr{\url{http://www.lewark.de/lukas/}}
\author{Andrew Lobb}
\address{Mathematical Sciences, Durham University, UK}
\email{\myemail{andrew.lobb@durham.ac.uk}}
\urladdr{\url{http://www.maths.dur.ac.uk/users/andrew.lobb/}}
\usepackage{thmtools,enumerate,graphicx,xcolor,enumitem,afterpage,float,ucs,amsmath,amssymb,amscd,url,mathtools}
\usepackage{microtype}
\usepackage{tikz-cd}
\tikzset{
	symbol/.style={
		draw=none,
		every to/.append style={
			edge node={node [sloped, allow upside down, auto=false]{$#1$}}}
	}
}
\usepackage{pinlabel}
\usepackage[latin1]{inputenc}
\usepackage[all]{xy}
\usepackage{amsthm}
\usepackage[pagebackref]{hyperref}
\usepackage[nameinlink,sort]{cleveref}
\let\cref\Cref
\crefname{subsection}{subsection}{subsections}
\Crefname{subsection}{Subsection}{Subsections}
\Crefname{enumi}{}{}
\crefname{equation}{}{}
\definecolor{darkblue}{RGB}{0,0,96}
\definecolor{gray}{RGB}{127,127,127}
\definecolor{darkred}{RGB}{160,0,0}
\definecolor{lightyellow}{RGB}{255,255,128}
\hypersetup{colorlinks={true},linkcolor={black},citecolor={black},filecolor={black},urlcolor={black},pdfauthor={\authors},pdftitle={\shorttitle}}
\newcommand{\myemail}[1]{\href{mailto:#1}{#1}}
\newcommand{\qua}{\hskip 0.4em \ignorespaces}
\def\arxiv#1{\relax\ifhmode\unskip\qua\fi
\href{http://arxiv.org/abs/#1}%
{\tt arXiv:\penalty -100\unskip#1}}

\def\MR#1{\relax\ifhmode\unskip\qua\fi
\href{http://www.ams.org/mathscinet-getitem?mr=#1}{\tt MR#1}}
\def\xox#1{\csname xx#1\endcsname}

\renewcommand*{\backrefalt}[4]{%
\tiny
\ifcase #1 %
No citations.%
\or
Cited on page~#2.%
\else
Cited on pages~#2.%
\fi
}

\newcommand{\myqed}{\pushQED{\qed}\qedhere}

\declaretheorem[numberwithin=section]{lemma}

\newtheorem{corollary}[lemma]{Corollary}
\newtheorem{proposition}[lemma]{Proposition}

\newtheorem*{prize*}{Prize}
\newtheorem*{theorem*}{Theorem}
\newtheorem{question}[lemma]{Question}
\newtheorem*{question*}{Question}
\theoremstyle{definition}
\newtheorem{definition}[lemma]{Definition}

\newtheorem{remark}[lemma]{Remark}

\newtheorem{example}[lemma]{Example}
\DeclareMathOperator{\gl}{gl}

\DeclareMathAlphabet{\mathpzc}{OT1}{pzc}{m}{it}
\DeclareMathOperator{\sgn}{sgn}

\DeclareMathOperator{\disc}{disc}
\newcommand{\cC}{\mathcal{C}}

\newcommand{\Z}{\mathbb{Z}}
\newcommand{\C}{\mathbb{C}}

\newcommand{\R}{\mathbb{R}}

\renewcommand{\det}{\text{det}}

\newcommand{\ls}{\ell\hspace{-.5pt}s}
\newcommand{\ds}{d\hspace{-.5pt}s}
\newcommand{\sss}{s\hspace{-.5pt}s\hspace{-.5pt}s}

\newcommand{\hash}{\#}

\newcommand{\TE}{T\hspace{-1pt}E}
\def\sqp{strongly quasi\-positive}
\def\sqn{strongly quasi\-negative}
\def\sqh{strongly quasi\-homo\-geneous}

\def\qh{quasi\-homo\-geneous}

\def\Qh{Quasi\-homogeneous}

\graphicspath{{figures/}}

\begin{document}
\thispagestyle{empty}
\begin{abstract}
Squeezed knots are those knots that appear as slices of genus-minimizing oriented smooth cobordisms between positive and negative torus knots.  We show that this class of knots is large and discuss how to obstruct squeezedness.  The most effective obstructions appear to come from quantum knot invariants, notably including refinements of the Rasmussen invariant due to Lipshitz-Sarkar and Sarkar-Scaduto-Stoffregen involving stable cohomology operations on Khovanov homology.

\end{abstract}
\maketitle
\section{Introduction}
\label{sec:intro}
\begin{definition}
	\label{defn:sqzd}
	We call a knot $K \subset S^3$ \emph{squeezed} if and only if it appears as a slice of a genus-minimizing oriented connected smooth cobordism between a positive torus knot $T^+$ and a negative torus knot $T^-$.  If so, we further say that $K$ is \emph{squeezed between} $T^+$ and~$T^-$.
\end{definition}

\begin{figure}[ht]
	\labellist
	\small
	\pinlabel {$T^+ \subset S^3 \times \{ 1 \}$} at 1700 1900
	\pinlabel {$T^- \subset S^3 \times \{ -1 \}$} at 1800 300
	\pinlabel {$K \subset S^3 \times \{ 0 \}$} at 1700 1100
	\endlabellist
	\includegraphics[scale=0.068]{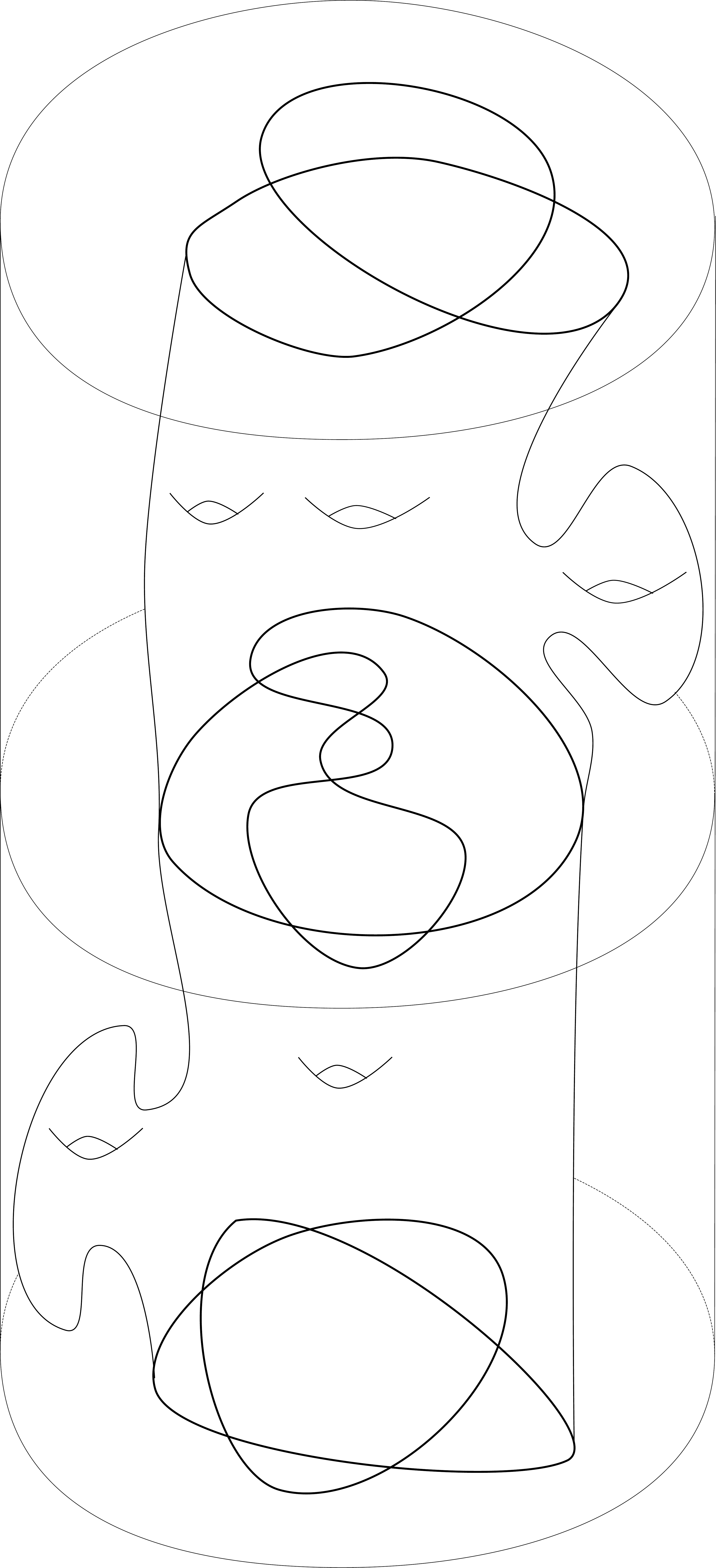}
	\caption{We draw a schematic of a genus-minimizing knot cobordism $\Sigma \subset S^3 \times [-1,1]$ between the positive torus knot $T^+$ and the negative torus knot $T^-$.  The knot $K$ appears a slice of this knot cobordism and hence $K$ is a squeezed knot.}\label{fig:sqzd_definition}
\end{figure}

The world-weary knot theorist will doubtless have inwardly groaned on reading this definition.  Why after all should they care about a new class of knots when there already exist so many interesting classes?  For example there are positive knots, almost positive knots, strongly quasipositive knots, quasipositive knots, all the negative counterparts, alternating knots, alternative knots \cite{kauffman}, %
homogeneous knots \cite{cromwell}, and pseudoalternating knots \cite{MR402718}.  The reader may perk up on reading the following proposition.

\begin{proposition}
	\label{prop:all_those_classes_are_sqzd}
	If $K$ is positive, almost positive, strongly quasipositive, quasipositive, any of the negative counterparts, alternating, alternative, %
	homogeneous, or pseudoalternating, then $K$ is squeezed.  Furthermore, the concordance classes of squeezed knots form a subgroup of the smooth concordance group.
\end{proposition}

The simplest example of a squeezed knot not immediately obvious from \cref{defn:sqzd} is the Figure Eight knot. It is squeezed between the positive and the negative trefoil knots, as illustrated in \cref{fig:figure_eight_example}.

\begin{figure}[ht]
	\includegraphics[scale=0.02]{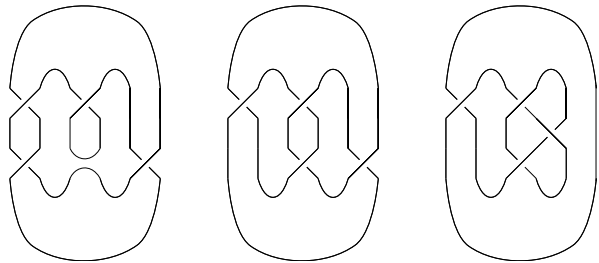}
	\caption{On the left is the positive trefoil knot, and on the right is the negative trefoil knot.  The knot in the middle is the Figure Eight knot, which is obtained from either trefoil by the addition of two oriented $1$-handles.  Hence we have exhibited the Figure Eight knot as a slice of a genus $2$ (which is minimal) cobordism between the positive and negative trefoils.}
	\label{fig:figure_eight_example}
\end{figure}

Of course, the above subsumption of classes of knots is only interesting if not \emph{every} knot is squeezed.

\begin{proposition}
	\label{prop:six_knots_of_ten_or_fewer_crossings}
	Of the $249$ prime knots of $10$ or fewer crossings, at least $243$ are squeezed and at least $4$ are not squeezed.  In fact, the knots $9_{42}$, $10_{125}$, $10_{132}$, and $10_{136}$ are not squeezed.
\end{proposition}

\begin{prize*}
	\label{prize:squeezedness_status_of_remaining_two_knots}
	We offer $130$ Swiss Francs and $141$ Swiss Francs for determining the \mbox{squeezedness} status of each of the knots $10_{130}$ and $10_{141}$, respectively.
\end{prize*}

Among those $243$ squeezed knots, not every knot lies in one of the classes covered by \cref{prop:all_those_classes_are_sqzd}.  Each of the $243$ is however what we shall call \emph{quasihomogeneous} -- although we leave the definition of this concept until later.

Since the definition of a squeezed knot involves smooth knot cobordism, one might predict that one most easily obstructs squeezedness by co-opting an invariant from the suite of Floer homological invariants, which have been so powerful in answering questions about knot cobordism and concordance.  In fact, the obstructions that we have found to be useful on small knots rather arise from the quantum side of knot theory.

Maybe most strikingly, the only way we have found to obstruct the knots $9_{42}$, $10_{132}$, and $10_{136}$ from being squeezed is to use refinements of the Rasmussen invariant using stable cohomology operations.  In particular, we appeal to the spacification of Khovanov homology due to Lipshitz-Sarkar \cite{LipSar4} and to their refinement of the Rasmussen invariant \cite{LipSar3} using the second Steenrod square, or we appeal to the refinement of the Rasmussen invariant using the first Steenrod square on odd Khovanov homology due to Sarkar-Scaduto-Stoffregen \cite{MR4078823}.

These obstructions arise since the concordance invariants coming from quantum knot invariants tend to be boring on squeezed knots in a sense made precise in the following proposition.

\begin{proposition}
	\label{prop:boring_on_squeezed}
	If $K$ is a squeezed knot, then the value of the following invariants on $K$ depends only on the Rasmussen invariant $s(K)$.
	\begin{itemize}
		\item All slice torus invariants. In fact, they all take the value $s(K)/2$.
		\item The Lipshitz-Sarkar refinements of the Rasmussen invariant.
		\item The Sarkar-Scaduto-Stoffregen refinements of the Rasmussen invariant.
                \item Sch\"utz's integral Rasmussen invariant $s^{\mathbb{Z}}$.
		\item The refinements $S_n$ and $\gimel_n$ of the $\mathfrak{sl}_n$ slice torus invariants.
	\end{itemize}
\end{proposition}

Typical examples of slice torus invariants (see \cref{defn:slice_torus}) are suitable normalizations of the Rasmussen invariant $s$, the $\tau$ invariant, or the $\mathfrak{sl}_n$ concordance invariants $s_n$.
That the knot $10_{125}$ is not squeezed can be deduced by comparing $s(10_{125})$ to $s_3(10_{125})$.
Note that since $10_{125}$ is quasialternating, it cannot be obstructed from being squeezed by comparing $s(10_{125})$ with $\tau(10_{125})$~\cite{manozs}.

\begin{remark}\label{rmk:g2}
While many invariants are determined for squeezed knots,
this is in general not true for the \emph{slice genus} $g_4(K)$ of a knot $K$,
which is defined as the minimal genus realized by oriented connected smooth cobordisms between $K$ and the unknot.
For example, consider a knot $K$ squeezed between the positive and the negative trefoil.  The minimal genus of a cobordism from the positive to negative trefoil is $2$, while each trefoil has slice genus $1$.  From this we can deduce that $g_4(K) \in \{0,1,2\}$. All these values are assumed for different
choices of $K$: for example take the unknot, the figure-eight knot (see \cref{fig:figure_eight_example}), and the knot $J$ shown in \cref{fig:g2}, respectively. See \Cref{ex:theknotJ} for more details on $J$.
\end{remark}

\begin{figure}[th]
\includegraphics[width=2in,height=2in]{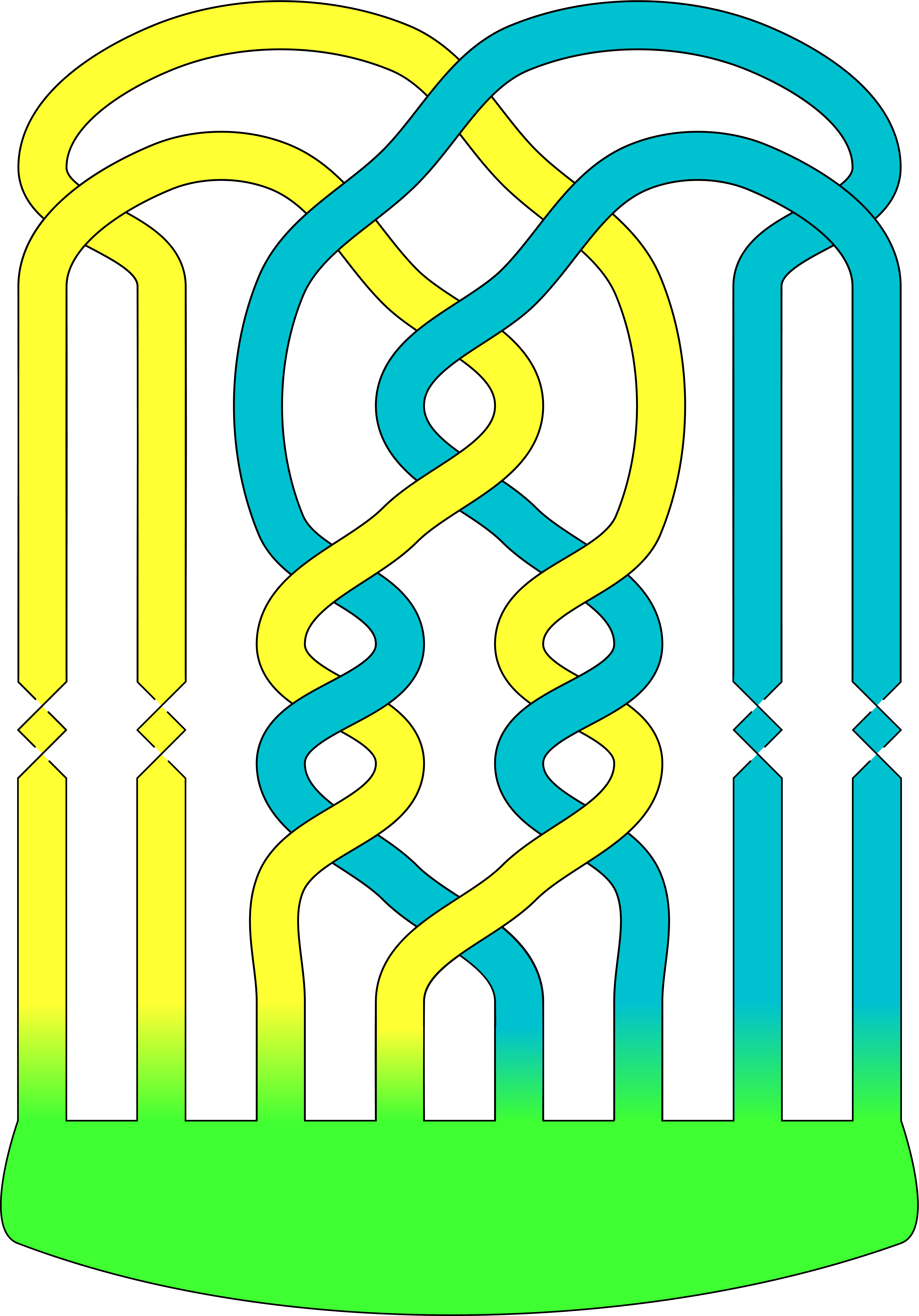}
\caption{A strongly quasihomogeneous knot with slice genus $2$, squeezed between the positive and the negative trefoil.%
}
\label{fig:g2}
\end{figure}

\subsection*{Plan of the paper}

\cref{sect:defs} contains definitions and topological constructions we use to verify that the 243 knots of \cref{prop:six_knots_of_ten_or_fewer_crossings} are squeezed, and to prove \cref{prop:all_those_classes_are_sqzd}.  The obstructions to being squeezed are explained and applied in \cref{sect:obs} to the four non-squeezed knots of \cref{prop:six_knots_of_ten_or_fewer_crossings}.  In \cref{sect:extra} we explore some natural directions that arise from our study of squeezedness.

\subsection*{Acknowledgments}The first author gratefully acknowledges support by the SNSF Grant 181199.
The second author gratefully acknowledges support by the DFG via the Emmy Noether Programme, project no.~412851057.
We thank Robert Lipshitz and Dirk Sch\"utz for insightful comments on an early version of this paper that improved the presentation, in particular Dirk for pointing us to \cite{MR4078823} and \cite{schuetzodd}.
We thank the anonymous referee for their thoughtful remarks.

\section{Definitions, examples, and constructions}
\label{sect:defs}
All our manifolds will be smooth, compact, and oriented.  Although our main interest is knots, %
we widen our scope to include links, as this will enable us to derive a good criterion for determining squeezedness.

\begin{definition}
	\label{defn:minimal_cobordism}
	A connected cobordism between two links is called \emph{genus-minimizing}, if and only if it realizes the minimal genus (or equivalently, maximal Euler characteristic) among all connected cobordisms between the two links.
\end{definition}

Recall from \cref{defn:sqzd} that a squeezed knot is a knot which appears as a slice of a genus-minimizing cobordism between a positive and a negative torus knot. We begin this section by giving equivalent formulations of \cref{defn:sqzd}.

\begin{proposition}\label{prop:equivDefSqueezed}
Let $K$ be a knot. The following statements are equivalent.
\begin{itemize}
\item $K$ is squeezed.
\item There exists a genus-minimizing cobordism $C\subset S^3\times[-1,1]$ between a strongly quasipositive link and a strongly quasinegative link  such that\break $K\times\{0\}=C\pitchfork S^3\times\{0\}$.
\item There exists a genus-minimizing cobordism $C\subset S^3\times[-1,1]$ between a quasipositive link and a quasinegative link such that $K\times\{0\}=C\pitchfork S^3\times\{0\}$.
\end{itemize}
\end{proposition}

To readers familiar with the definitions of strong quasipositivity and quasipositivity, it will be clear that each bullet point implies the succeeding bullet point. We add one clarification: we will always assume that (strong) quasipositive and (strong) quasinegative surfaces are connected, and hence we will in particular \emph{not} consider unlinks with more than one component to be quasipositive. For those readers unfamiliar, we shall set out the relevant definitions below.  The meat of our proof of \cref{prop:equivDefSqueezed} is then in showing that the third bullet point implies the first bullet point; but this is
essentially of the same level of difficulty as
showing that the slice-Bennequin inequality is an equality for quasipositive braids using that it is an equality for the standard torus knot braids~\cite[Sec.~3]{rudolph_QPasObstruction}.

It is the main work of this section to define the notion of \emph{quasihomogeneity}.  It is a definition that is motivated by the third bullet point of \cref{prop:equivDefSqueezed}  -- quasihomogeneous knots are boundaries of immersed surfaces which represent a genus-minimizing cobordism as in this bullet point.  In \cref{subsec:quasihom_def} we define quasihomogeneity modulo some details about quasipositivity and establish that quasihomogeneous knots are squeezed.  In \cref{subsec:quasipos} we fill in the missing details and give an example.  Then in \cref{subsec:10-crossing} we show that the 243 knots of \cref{prop:six_knots_of_ten_or_fewer_crossings} are squeezed.  Finally, in \cref{subsec:sqzd_knots_are_fucking_everywhere}, quasihomogeneous knots will be shown to subsume all the classes of knots from \cref{prop:all_those_classes_are_sqzd}, thus proving that proposition.

\subsection{Ribbon surfaces and quasihomogeneity}
\label{subsec:quasihom_def}
Generically immersed surfaces with boundary in $S^3$ contain $1$-manifolds of double points and $0$-manifolds of triple points.  A ribbon surface is an immersed surface with no triple points in which the $1$-manifolds of double points have a prescribed form:

\begin{definition}
	\label{def:ribbon_surfaces}
	A \emph{ribbon surface} is an immersed orientable compact surface without closed components in $S^3$ such that the only self-intersections of the surface are \emph{ribbon self-intersections} as defined in \cref{fig:ribbon_intersection}.  We consider the \emph{boundary} of the ribbon surface as a link in $S^3$.
\end{definition}

\begin{figure}[b]
	\includegraphics[scale=0.07]{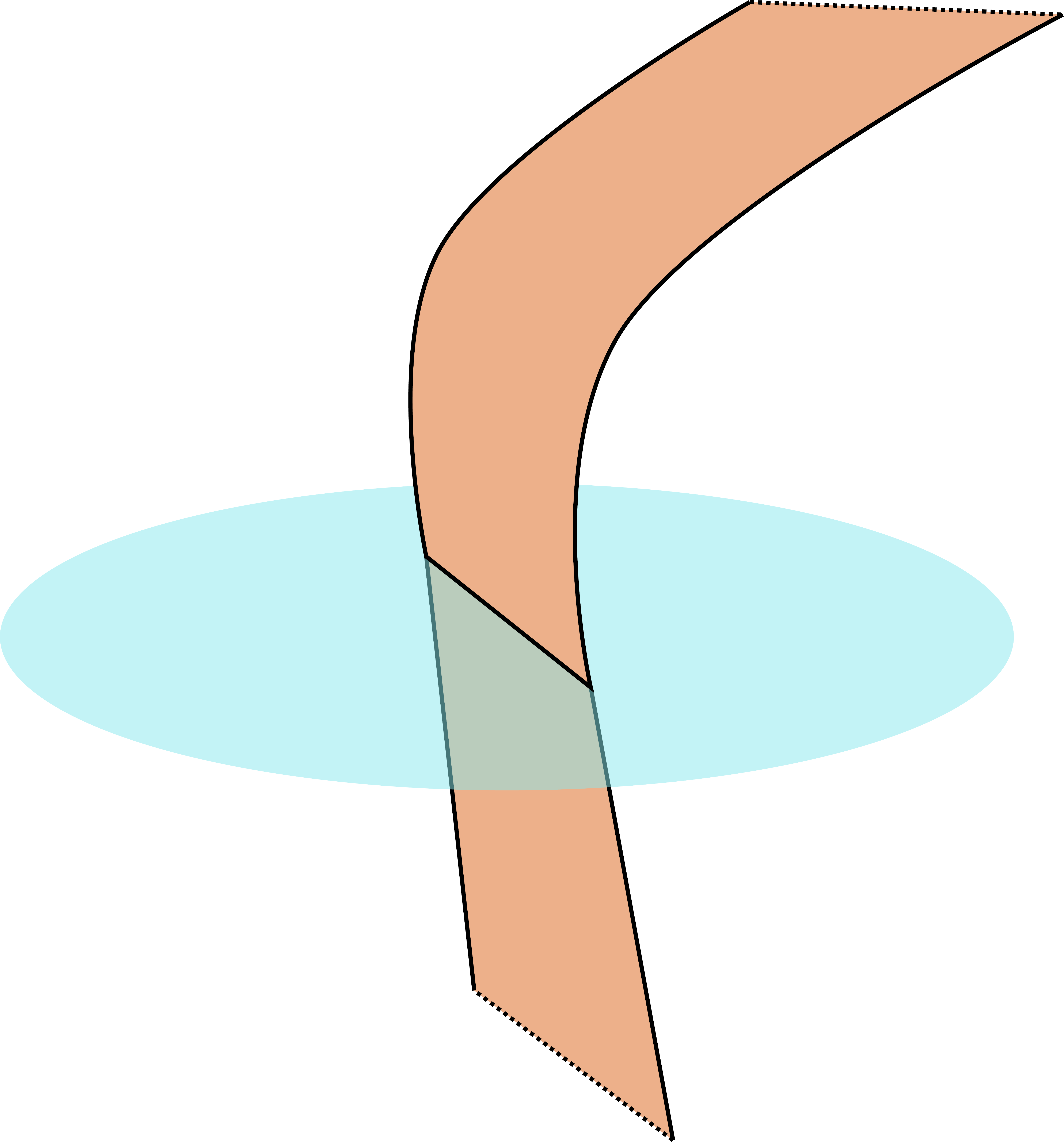}
	\caption{We show a $1$-manifold of intersection in an immersed surface in the $3$-sphere.  Such self-intersections (where the preimage of the locus of the intersection consists of an arc with both boundary points on the boundary of the surface, and an arc in the interior of the surface) are called \emph{ribbon self-intersections}.}\label{fig:ribbon_intersection}
\end{figure}

\begin{figure}[b]
	\includegraphics[scale=0.07]{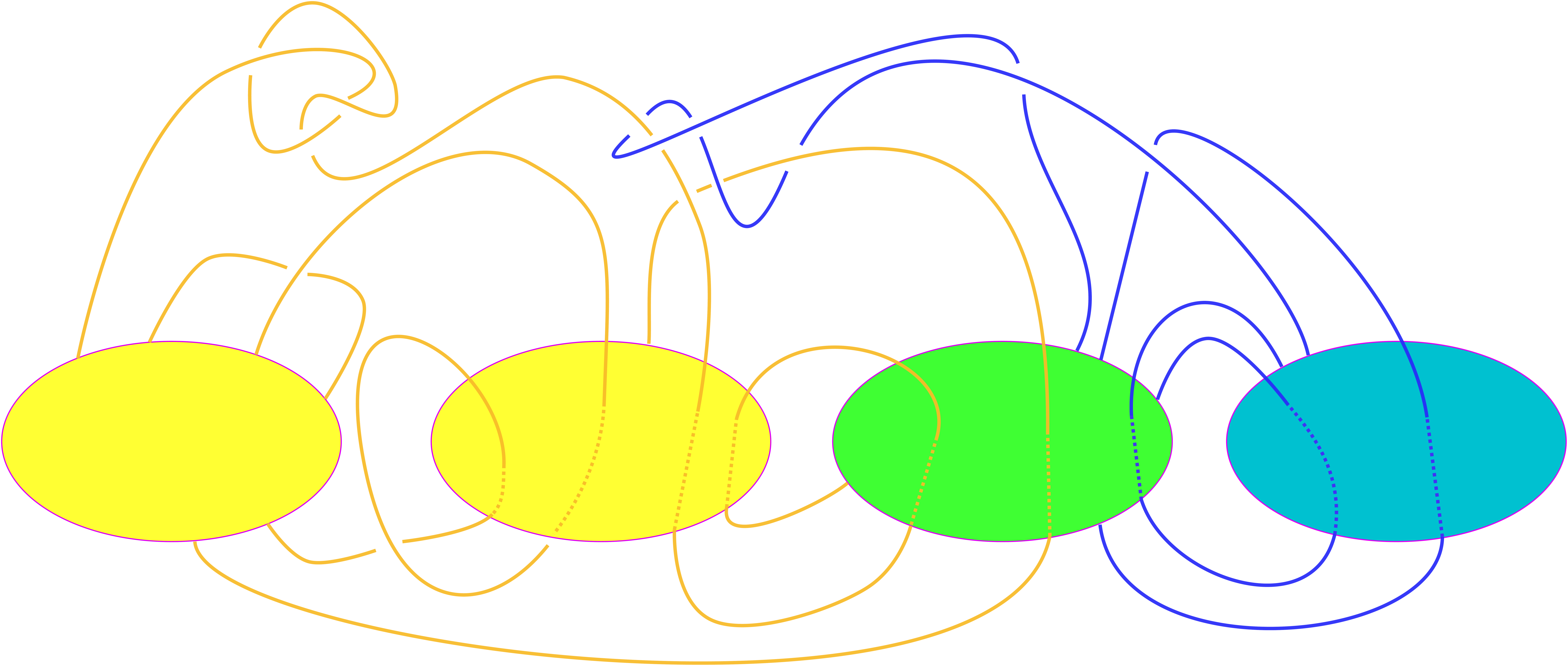}
	\caption{A motivational picture for the definition of quasihomogeneous.
		The green disc together with the yellow discs and the yellow ribbons should form a quasinegative surface, while the green disc together with the blue discs and the blue ribbons should form a quasipositive surface.  The blue ribbons do not intersect the yellow discs, while the yellow ribbons do not intersect the blue discs.  We refer to the union of all the discs and ribbons as a \emph{quasihomogeneous surface}.}
	\label{fig:qh_motivation_picture}
\end{figure}
It is straightforward to see that one may describe any ribbon surface up to ambient isotopy in the following way.  Start with a collection of disjoint embedded discs in $S^3$ (for example the yellow, green, and blue discs in \cref{fig:qh_motivation_picture}).  Then add some disjoint embedded arcs with endpoints on the boundaries of the discs (such as the yellow and blue arcs of \cref{fig:qh_motivation_picture}).  The interiors of the arcs are allowed to transversely intersect the interiors of the embedded discs.  Then, by thickening each arc to a \emph{ribbon}, obtain a ribbon surface.
There is some choice in how to thicken the arcs, given by a framing of each arc relative to its endpoints.
The requirement that the ribbon surface be orientable results in certain modulo 2 restrictions on these framings.
Note that the ribbon singularities occur exactly where the ribbons intersect the interiors of the discs.

A ribbon surface determines a smoothly embedded surface in the $4$-ball whose boundary is the boundary of the ribbon surface. Here is an explicit upside down handle description. Start with an interval times the boundary of the ribbon surface and add a $1$-handle for each ribbon (dual to the arc at the core of the ribbon).  The resulting link is isotopic to the boundaries of the discs that we started with. A~$2$-handle for each such disc can then be added to complete the description.

\begin{definition}
	\label{def:qp_sqp}
	A \emph{quasipositive} link $L$ is a link that appears as the boundary of a \emph{quasipositive surface}
	(similarly for quasinegative).  A \emph{strongly quasipositive} link $L$ is a link that appears as the boundary of an \emph{embedded} quasipositive surface (similarly for strongly quasinegative).
\end{definition}

We postpone the definition of quasipositive and quasinegative surfaces. It is more important at this point for us to know the following. They are ribbon surfaces (in particular they are immersed in $S^3$ as described in~\cref{def:ribbon_surfaces}) for which the following three facts hold.
For the reader familiar with quasipositivity, we highlight the first bullet point, which is often not assumed in the literature.

\begin{lemma}
	\label{lem:qp_facts}
	We have the following.
	
	\begin{itemize}
	\item Quasipositive and quasinegative surfaces have the property that they are images of immersions of \emph{connected} surfaces.
	\item If $L$ is the boundary of a quasipositive or quasinegative surface then the surface describes a genus-minimizing cobordism from the link $L$ to the empty link.
	\item If $L^+$ is the boundary of a quasipositive surface $\Sigma^+$, $L^-$ is the boundary of a quasinegative surface $\Sigma^-$, and $\Sigma$ is a connected cobordism from $L^+$ to $L^-$ of maximal Euler characteristic,
then $\chi(\Sigma) = \chi(\Sigma^+) + \chi(\Sigma^-) - 2$.
	\item The connected sum $K_1 \# K_2$ of two quasipositive (respectively quasinegative) knots $K_1$ and $K_2$ is quasipositive (respectively quasinegative), and $g_4(K_1 \# K_2) = g_4 (K_1) + g_4(K_2)$.
	\end{itemize}
\end{lemma}
\noindent With this in hand we are ready to define the main concept of this section -- that of \emph{quasihomogeneity}.

\begin{definition}
	\label{def:qh}
	A \emph{quasihomogeneous} knot is a knot which is the boundary of a quasihomogeneous surface.  The reader should refer to \cref{fig:qh_motivation_picture} and its caption.
	
	A \emph{quasihomogeneous surface} is a ribbon surface in $S^3$ 
	composed of discs and ribbons as discussed after \Cref{def:ribbon_surfaces}, with the following additional requirements. It is comprised of a single green disc; some blue discs and some yellow discs; and some yellow ribbons and some blue ribbons.  The yellow ribbons do not meet the blue discs anywhere and the blue ribbons do not meet the yellow discs anywhere (including at the attaching points of the ribbons).  The union of the green disc, yellow discs, and yellow ribbons should be a quasinegative surface, while the union of the green disc, blue discs, and blue ribbons should be a quasipositive surface.
	
	If the quasihomogeneous ribbon surface is embedded, i.e.~it is a Seifert surface, then we call it a \emph{\sqh}\ surface and the knot a \emph{strongly quasihomogeneous} knot.
\end{definition}
\begin{remark}
From the above definition, it is clear that a quasihomogeneous surface is the union of a quasipositive surface and a quasinegative surface that intersect in a disc.
Similarly, a strongly quasihomogeneous surface is the union of a strongly quasipositive surface and a strongly quasinegative surface that intersect in a disc.
\end{remark}
\begin{example}\label{ex:surffor162163164}\Cref{fig:10_163,fig:10_162} provide strongly quasihomogeneous and \qh\ surfaces, respectively, with boundary prime knots with crossing number~10.
\end{example}
\begin{figure}
\def\svgscale{1.8}
\begingroup%
  \makeatletter%
  \providecommand\color[2][]{%
    \errmessage{(Inkscape) Color is used for the text in Inkscape, but the package 'color.sty' is not loaded}%
    \renewcommand\color[2][]{}%
  }%
  \providecommand\transparent[1]{%
    \errmessage{(Inkscape) Transparency is used (non-zero) for the text in Inkscape, but the package 'transparent.sty' is not loaded}%
    \renewcommand\transparent[1]{}%
  }%
  \providecommand\rotatebox[2]{#2}%
  \newcommand*\fsize{\dimexpr\f@size pt\relax}%
  \newcommand*\lineheight[1]{\fontsize{\fsize}{#1\fsize}\selectfont}%
  \ifx\svgwidth\undefined%
    \setlength{\unitlength}{195.75121542bp}%
    \ifx\svgscale\undefined%
      \relax%
    \else%
      \setlength{\unitlength}{\unitlength * \real{\svgscale}}%
    \fi%
  \else%
    \setlength{\unitlength}{\svgwidth}%
  \fi%
  \global\let\svgwidth\undefined%
  \global\let\svgscale\undefined%
  \makeatother%
  \begin{picture}(1,0.3623333)%
    \lineheight{1}%
    \setlength\tabcolsep{0pt}%
    \put(0,0){\includegraphics[width=\unitlength,page=1]{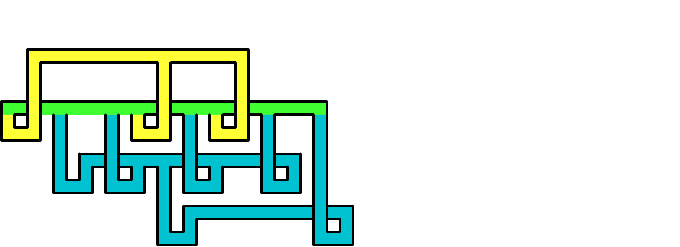}}%
    \put(-0.00649637,0.30834807){\color[rgb]{0,0,0}\makebox(0,0)[lt]{\lineheight{1.25}\smash{\begin{tabular}[t]{l}$10_{163}$\end{tabular}}}}%
    \put(0,0){\includegraphics[width=\unitlength,page=2]{10_163_10_164.pdf}}%
    \put(0.5501737,0.30834808){\color[rgb]{0,0,0}\makebox(0,0)[lt]{\lineheight{1.25}\smash{\begin{tabular}[t]{l}$10_{164}$\end{tabular}}}}%
    \put(0,0){\includegraphics[width=\unitlength,page=3]{10_163_10_164.pdf}}%
  \end{picture}%
\endgroup%

\caption{Strongly quasihomogeneous surfaces.
\newline Left: A Murasugi sum of a \sqp\ surface (green union blue) and a \sqn\ surface (green union yellow) with boundary $10_{163}$.
\newline Right: A \sqh\ surface for $10_{164}$. %
}
\label{fig:10_163}
\end{figure}
\begin{figure}
\def\svgscale{2.6}
\begingroup%
  \makeatletter%
  \providecommand\color[2][]{%
    \errmessage{(Inkscape) Color is used for the text in Inkscape, but the package 'color.sty' is not loaded}%
    \renewcommand\color[2][]{}%
  }%
  \providecommand\transparent[1]{%
    \errmessage{(Inkscape) Transparency is used (non-zero) for the text in Inkscape, but the package 'transparent.sty' is not loaded}%
    \renewcommand\transparent[1]{}%
  }%
  \providecommand\rotatebox[2]{#2}%
  \newcommand*\fsize{\dimexpr\f@size pt\relax}%
  \newcommand*\lineheight[1]{\fontsize{\fsize}{#1\fsize}\selectfont}%
  \ifx\svgwidth\undefined%
    \setlength{\unitlength}{79.51035447bp}%
    \ifx\svgscale\undefined%
      \relax%
    \else%
      \setlength{\unitlength}{\unitlength * \real{\svgscale}}%
    \fi%
  \else%
    \setlength{\unitlength}{\svgwidth}%
  \fi%
  \global\let\svgwidth\undefined%
  \global\let\svgscale\undefined%
  \makeatother%
  \begin{picture}(1,0.89204943)%
    \lineheight{1}%
    \setlength\tabcolsep{0pt}%
    \put(0,0){\includegraphics[width=\unitlength,page=1]{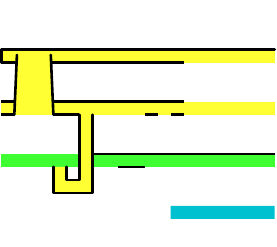}}%
    \put(-0.01321581,0.75913999){\color[rgb]{0,0,0}\makebox(0,0)[lt]{\lineheight{1.25}\smash{\begin{tabular}[t]{l}$10_{162}$\end{tabular}}}}%
    \put(0,0){\includegraphics[width=\unitlength,page=2]{10_162.pdf}}%
  \end{picture}%
\endgroup%

\caption{A quasihomogeneous surface $F$ with boundary the knot $10_{162}$ given as the union of a quasinegative surface (green union yellow) and a quasipositive (in fact, strongly quasipositive) surface (green union blue). There are two ribbon singularities (drawn as thin black lines). They arise when the leftmost hook-shaped ribbon that starts at the top yellow horizontal disc and ends at the green horizontal disc first intersects the top yellow horizontal disc and then self-intersects. Compare with the paragraph before \Cref{def:qp_sqp_braid_surface} for a description of $F$ using braids.
}
\label{fig:10_162}
\end{figure}

Quasihomogeneity gives us a good criterion for establishing the squeezedness of small knots.
\begin{proposition}
	\label{prop:qh_knots_are_squeezed}
	Quasihomogeneous knots are squeezed.
\end{proposition}

\begin{proof}
	Referring to \cref{fig:qh_motivation_picture} and its caption we see a ribbon surface that is the union of a quasipositive and quasinegative surface -- they meet in a single disc that we have colored green.  Let us refer to the quasipositive surface boundary as $L^+$, the quasinegative surface boundary as $L^-$, and the boundary of the whole ribbon surface as $K$.
	
	We describe a cobordism from $L^+$ to $L^-$ that has $K$ as a slice.  Starting with $L^+$, add the yellow disc and the yellow ribbons to arrive at $K$.  Then add the duals to the blue ribbons and cap off the blue discs to arrive at $L^-$.  A handle count reveals that the Euler characteristic of this cobordism is two less than the sum of the Euler characteristics of the quasipositive surface and the quasinegative surface. 
 Hence, by the third bullet point of \cref{lem:qp_facts}, it is a genus-minimizing cobordism 
and so, by \cref{prop:equivDefSqueezed}, $K$ is squeezed.
\end{proof}

\begin{example} The prime knot $10_{163}$ is not positive, not negative, not \sqp, not \sqn, not alternating, and not homogeneous.
However, 
$10_{163}$ is \sqh; see \cref{fig:10_163}. In particular, $10_{163}$ is a squeezed knot by \cref{prop:qh_knots_are_squeezed}.
\end{example}

\begin{example}\label{ex:theknotJ}
Let us have a closer look at the knot $J$ from \Cref{rmk:g2}. The Seifert surface $\Sigma$ of $J$ with genus 2 shown in \cref{fig:g2} is strongly quasihomogeneous.
Indeed, cutting the two yellow 1-handles (on the left) yields the genus one Seifert surface of the positive trefoil, which is a strongly quasipositive surface, while cutting the two blue 1-handles (on the right) yields the genus one Seifert surface of the negative trefoil, which is a strongly quasinegative surface.
So in particular, $J$ is squeezed between the positive and the negative trefoil.
Let us show $g_4(J) = 2$.
Taking the cores of the four obvious 1-handles as basis, one finds the following Seifert matrix $A$ of $\Sigma$:\\[1ex]
\hspace*{3cm}$\displaystyle
A=
\begin{pmatrix}
1 & 0 & -2 & -1 \\
1 & 1 & -1 & -2 \\
-2 & -1 & -1 & -1 \\
-1 & -2 & 0 & -1 \\
\end{pmatrix}.
$\\[1ex]
The associated quadratic form $\eta$ has discriminant $\disc \eta = \det(A + A^{\top}) = 225 = p^2\cdot (8\cdot p + 1)$ for $p = 3$,
and the Hasse symbol of $\eta$ over $\mathbb{Q}_p$ is $-1$. This implies by \cite[Theorem~1]{MR3938580} that there is no non-trivial solution for $v^{\top}(A + A^{\top})v = 0$ for $v \in \Z^4$,
which in turn implies that $g_4(J) \geq 2$ \cite{MR547461}.
\end{example}

\begin{example}\label{ex:sqh}All \sqp\ knots and all \sqn\ knots are \sqh.
If a knot $K$ is the connected sum of a \sqp\ knot and a \sqn\ knot, then $K$ is \sqh. %

More generally, any Seifert surface $F$ that arises as the Murasugi sum of a \sqp\ surface with strongly quasinegative surface is strongly quasihomogeneous; see the left-hand side of \cref{fig:10_163} for an example. Hence, knots that arise as the boundary of such $F$ are \sqh.
\end{example}
\begin{example}\label{ex:oddpretzels}
All 3-stranded pretzel knots $P(p,q,r)$ with $p,q,r$ odd integers are squeezed.
In fact $P(p,q,r)$ is slice, or its canonical Seifert surface is \sqh. 
Let us show that now. The canonical Seifert surface of the $P(p,q,r)$ pretzel
consists of two discs, connected by three bands with $p,q,r$ half twists, respectively.
Let us distinguish three exhaustive cases:
\begin{itemize}
\item If there is a positive and a negative integer among $p + q, q + r, r + p$, say w.l.o.g.\ $p + q > 0$ and $q + r < 0$, then the canonical Seifert surface $\Sigma$ of $P(p,q,r)$ is \sqh.
Indeed, it the union of one unknotted annulus with $p + q$ half twists, which is a strongly quasinegative surface,
and one unknotted annulus with $q + r$ half twists, which is a strongly quasipositive surface;
these two annuli intersect in a disc. In terms of 
\cref{def:qh}, $\Sigma$~can be decomposed into one green disc, one yellow ribbon and one blue ribbon.
\item If $p + q, q + r, r + p$ are all positive (negative), then the canonical Seifert surface of $P(p,q,r)$ is strongly quasinegative (strongly quasipositive), see \cite[Lemma~3]{rudolph_QPasObstruction} or \cite{MR1840734}. %
\item If one of $p + q, q + r, r + p$ is zero, then $P(p,q,r)$ is slice.
\end{itemize}
\end{example}
\begin{example}
For positive odd $p$, let $K_p$ be the $P(p+2,-p,p-2)$ pretzel knot.
This knot is squeezed between the negative trefoil and the positive trefoil (see \cref{ex:oddpretzels}).
Since the knot determinant $\det\,K_p = p^2 + 4$ is not a square, the Alexander polynomial
\[
\Delta_{K_p}(t) = \frac{p^2 + 3}{4}(t + t^{-1}) - \frac{p^2 + 5}{2}.
\]
of $K_p$ is irreducible. As a consequence of the Fox-Milnor condition and the fact that $\Z[t^{\pm 1}]$ is a UFD, each irreducible polynomial $f\in\Z[t^{\pm 1}]$ with $f(t) = f(t^{-1})$
yields a $\Z/2$-valued concordance invariant, which associates to a knot $K$
the parity of the maximal exponent $e\geq 0$ such that $f^e$ divides $\Delta_K$.
It follows that there is no pair of concordant knots in the family $K_1, K_3, K_5, \ldots $.
So, even the class of knots that are squeezed between a fixed positive and negative torus knot may be quite rich.
\end{example}

\subsection{Quasipositivity}
\label{subsec:quasipos}
We now have a definition of quasihomogeneity and have established that quasihomogeneous knots are squeezed, but we have postponed several definitions and proofs.  In this subsection we rectify this by doing the following.
\begin{itemize}
	\item We define quasipositivity and its relatives.
	\item Establish \cref{lem:qp_facts} giving relevant facts we have used about quasipositive links.
	\item We establish \cref{prop:equivDefSqueezed} giving an equivalent definition of squeezedness in terms of cobordisms between quasipositive and quasinegative links.
\end{itemize}
To define quasipositivity, we first need to spend a little time thinking about the braid group and ribbon surfaces.

Let us work in $B_n$, the braid group on $n$ strands generated by the standard Artin generators $a_1, \ldots, a_{n-1}$.  For each $j = 1, \ldots, k$ let $g_j$ be either an Artin generator or the inverse of an Artin generator.
Consider an element
\begin{equation}
\label{eq:r}
r = (g_{1} \cdots g_{k})  a_i^{\pm1}  (g_k^{-1} \cdots g_1^{-1}) {\rm ,}\end{equation} 
that is, an element which is a conjugate of a standard Artin generator or its inverse.

We recall a procedure to associate a ribbon surface to a product of such $r$. The closure of the trivial braid is the $n$-component unlink.  We think of this unlink as bounding $n$ disjoint discs.  The closure of the braid $r$ is an $(n-1)$-component link.  It is easy to see how $r$ determines a ribbon connecting two of the $n$ disjoint discs, and possibly intersecting the discs and itself in ribbon singularities.  We refer to~\cite[Sec.~2]{Rudolph_83_BraidedSurfaces} for details and only illustrate the construction with an example for $n=4$: \cref{fig:10_162} illustrates the ribbon surface associated to the following product of conjugates
\[
\underbrace{\raisebox{0pt}[0pt][3pt]{$\displaystyle a_1^{-2}a_2^{-1}a_1^2$}}_{r_1}
\underbrace{\raisebox{0pt}[0pt][3pt]{$\displaystyle a_2^{-1}$}}_{r_2}
\underbrace{\raisebox{0pt}[0pt][3pt]{$\displaystyle a_2^{-1}$}}_{r_3}
\underbrace{\raisebox{0pt}[0pt][3pt]{$\displaystyle a_3$}}_{r_4}
\underbrace{\raisebox{0pt}[0pt][3pt]{$\displaystyle a_1^{-1}$}}_{r_5}
\underbrace{\raisebox{0pt}[0pt][3pt]{$\displaystyle a_2^{-1}$}}_{r_6}
\underbrace{\raisebox{0pt}[0pt][3pt]{$\displaystyle a_3$}}_{r_7}.
\]
In more detail, the ribbon surface is constructed from four horizontally placed discs (yellow, yellow, green, and blue from top to bottom) that are connected by (hook-shaped) ribbons corresponding, from left to right, to $r_1$; $r_2$; $r_3$; $r_4$ and $r_5$; $r_6$; and $r_7$, respectively, where the leftmost ribbon has one ribbon singularity with itself and one with the top disc while the other ribbons are embedded.

We note that in general a ribbon associated with such a braid word $r$ is rather different than a ribbon obtained from thickening an arc as described in the paragraph after \Cref{def:ribbon_surfaces}, since the ribbon associated with $r$ can self-intersect. To arrive at such a description, the self-intersecting ribbons would have to be decomposed as unions of ribbons and discs. 

\begin{definition}
	\label{def:qp_sqp_braid_surface}
	Suppose that $b \in B_n$ is a braid word which is written as a product of conjugates of Artin generators.  This expression for $b$ determines a ribbon surface for the braid closure of $b$.  We call such a braid word \emph{quasipositive}.
	
	When such a ribbon surface is the image of an immersion of a connected surface, we call it \emph{quasipositive} and its boundary a \emph{quasipositive link}.  If the quasipositive surface is the image of an embedding, then we call it \emph{strongly quasipositive} and its boundary a \emph{strongly quasipositive link}.
	
	Quasinegative and strongly quasinegative surfaces and links are defined similarly but replacing products of conjugates of Artin generators with products of conjugates of inverses of Artin generators.
\end{definition}

\begin{proof}[Proof of \cref{lem:qp_facts}.]

The first bullet point is by definition.
	
	For the second bullet point, suppose that $L$ is the closure $\hat{b}$ of a quasipositive braid word $b \in B_n$.  Suppose that $b$ is the product of $R$ conjugates of Artin generators and consists of $R+P$ Artin generators and $P$ inverses of Artin generators. For illustration, consider e.g.~$b=b_{8_{20}} = (a_1a_2a_1^{-1}a_2a_1a_2^{-1}a_1^{-1})(a_2a_1a_2^{-1})$, where $n=3$, $R=2$, $P=4$, and $L=8_{20}$. Now consider the operation of inserting an Artin generator somewhere into the braid word $b$ to result in the braid word $b'$.  There is a cobordism between the closures $\hat{b}$ and $\hat{b'}$ consisting of a single $1$-handle.  Our approach is to do this repeatedly.
	
	First insert $P$ Artin generators into $b$ (viewed as a braid word), one next to each inverse Artin generator of $b$, so that the result, after cancellation, is a positive braid word -- precisely, a product of $R+P$ Artin generators. In case of $b_{8_{20}}$, we have 
	\[a_1a_2a_1^{-1}{\color{red}a_1}a_2a_1a_2^{-1}{\color{red}a_2}a_1^{-1}{\color{red}a_1}a_2a_1a_2^{-1}{\color{red}a_2}=a_1a_2a_2a_1a_2a_1,\] where added generators are marked red. By further inserting $p(n-1) - (R+P)$ Artin generators for some $p > 0$, we can then arrive at a positive braid word $t$ whose closure is a positive torus knot $\hat{t} = T_{n,p}$. In case of $b_{8_{20}}$, we can e.g.~insert two generators (red) as follows
	\[a_1a_2{\color{red}a_1}a_2a_1a_2a_1{\color{red}a_2}=t.\]
	In this case, $p=4$ and $t$ is the standard $3$-braid with closure $T_{3,4}$.

	We thus have constructed a cobordism between $T_{n,p}$ and $\hat{b}$ which consists of $p(n-1) - R$ $1$-handles and hence has Euler characteristic $R - p(n-1)$.

	Suppose there is a connected cobordism from $\hat{b}$ to the empty link of Euler characteristic $E$.  Then by composing this cobordism with that between $T_{n,p}$ and $\hat{b}$ we get a cobordism from $T_{n,p}$ to the empty link.  We know that the maximal Euler characteristic of such a cobordism is $1 - (n-1)(p-1)$ \cite{thom}, so we have
	\[ E + R - p(n-1) \leq 1 - (n-1)(p-1) \]
	and hence
	\[ E \leq n - R {\rm .} \]
	
	But the quasipositive surface for $\hat{b}$ determines a connected cobordism from $\hat{b}$ to the empty link of Euler characteristic $n-R$, and hence describes a genus-minimizing cobordism from $\hat{b}$ to the empty link.
	
	The third bullet point can be established following a similar line of argument.  Given $L^+$ and $L^-$ we realize them as above as slices of genus-minimizing cobordisms between a positive (respectively negative) torus knot and the empty link.  These cobordisms can then be punctured and joined together along the puncture thus giving a cobordism from a positive torus knot to a negative torus knot which factors through a cobordism from $L^+$ to $L^-$.  The cobordism between the torus knots is genus-minimizing
	so the cobordism between $L^+$ and $L^-$ is genus-minimizing as well.
	
	For the final bullet point, suppose that $K_i$ is the braid closure of the quasipositive braid word $b_i \in B_{n_i}$ for $i = 1,2$.  Let $i_1 \colon B_{n_1} \hookrightarrow B_{n_1 + n_2 - 1}$ be the inclusion on the first $n_1$ strands, and $i_2 \colon B_{n_2} \hookrightarrow B_{n_1 + n_2 -1}$ be the inclusion on the final $n_2$ strands.  Then the braid word $i_1(b_1) \cdot i_2(b_2)$ is quasipositive and its closure is $K_1 \hash K_2$.  A handle count yields the formula for the genus.
\end{proof}

\begin{proof}[Proof of \cref{prop:equivDefSqueezed}.]
	Suppose that $K$ is a knot which occurs as a slice of a genus-minimizing cobordism $\Sigma$ between the quasipositive link $L^+$ and the quasinegative link~$L^-$.
	
	The proof of \cref{lem:qp_facts} establishes that there is a genus-minimizing cobordism between some positive torus knot and some negative torus knot that factors through some genus-minimizing cobordism from $L^+$ to $L^-$.  But now all we need to observe is that replacing this cobordism factor by $\Sigma$  will still give a genus-minimizing cobordism between the torus knots, and $K$ is a slice of this cobordism.
\end{proof}

\subsection{Squeezing most knots up to 10 crossings}
\begin{table}[p]
\begin{tabular}{ c| c |c }
 knot & braid word &  \\ 
 \hline
 \rule{0pt}{2.5ex}   $9_{44}$ & $a_1\left(a_1a_1a_2a_1^{-1}a_1^{-1}\right)a_3^{-1}a_2a_3^{-1}$ & qh \\  
 \rule{0pt}{2.5ex}   $9_{48}$ & $a_1(a_1a_2a_1^{-1})a_2a_3^{-1}(a_1a_2a_1^{-1})a_2a_3^{-1}$ & sqh \\
 \rule{0pt}{2.5ex} $10_{135}$ & $a_1a_1(a_1a_2a_1^{-1})(a_2a_3^{-1}a_2^{-1})a_2^{-1}a_2^{-1}a_3^{-1}$ & sqh \\
 \rule{0pt}{2.5ex} $10_{144}$ & $(a_2^{-1}a_1a_2)(a_2^{-1}a_1a_2)(a_1^{-1}a_2a_1)a_3^{-1}a_2a_1a_3^{-1}$ & qh \\
 \rule{0pt}{2.5ex} $10_{146}$ & $a_1^{-1}a_1^{-1}a_2(a_1^{-1}a_2a_3^{-1}a_2^{-1}a_1)a_2a_2a_3^{-1}$ & qh \\
 \rule{0pt}{2.5ex} $10_{147}$ & $a_1^{-1}a_1^{-1}a_2a_1^{-1}(a_2a_3a_2^{-1})(a_1a_2^{-1}a_1^{-1})a_3$ & sqh \\
 \rule{0pt}{2.5ex} $10_{150}$ & $a_1a_1a_1a_2^{-1}a_1a_3(a_1a_2^{-1}a_1^{-1})a_3a_2$ & sqh \\
 \rule{0pt}{2.5ex} $10_{158}$ & $a_1^{-1}a_1^{-1}(a_1^{-1}a_2^{-1}a_1)a_3(a_1a_2a_1^{-1})a_2a_3$ & qh \\
 \rule{0pt}{2.5ex} $10_{160}$ & $a_1a_1a_1a_2a_1a_3^{-1}(a_1a_2a_1^{-1})a_2a_3^{-1}$ & sqh \\
 \rule{0pt}{2.5ex} $10_{162}$ & $(a_1^{-2}a_2^{-1}a_1^2)a_2^{-1}a_2^{-1}a_3a_1^{-1}a_2^{-1}a_3$ & qh \\
 \rule{0pt}{2.5ex} $10_{163}$ & $a_1^{-1}a_2a_2a_1^{-1}a_3a_2a_1^{-1}a_2(a_2a_3a_2^{-1})$ & sqh \\
 \rule{0pt}{2.5ex} $10_{164}$ & $a_1^{-1}a_2a_1^{-1}a_2(a_2a_3a_2^{-1})(a_1a_2^{-1}a_1^{-1})a_3$ & sqh \\
\end{tabular}
\vspace{2ex}
\caption{}\label{table:braids}
\end{table}
\begin{figure}[p]
\def\svgscale{0.95}
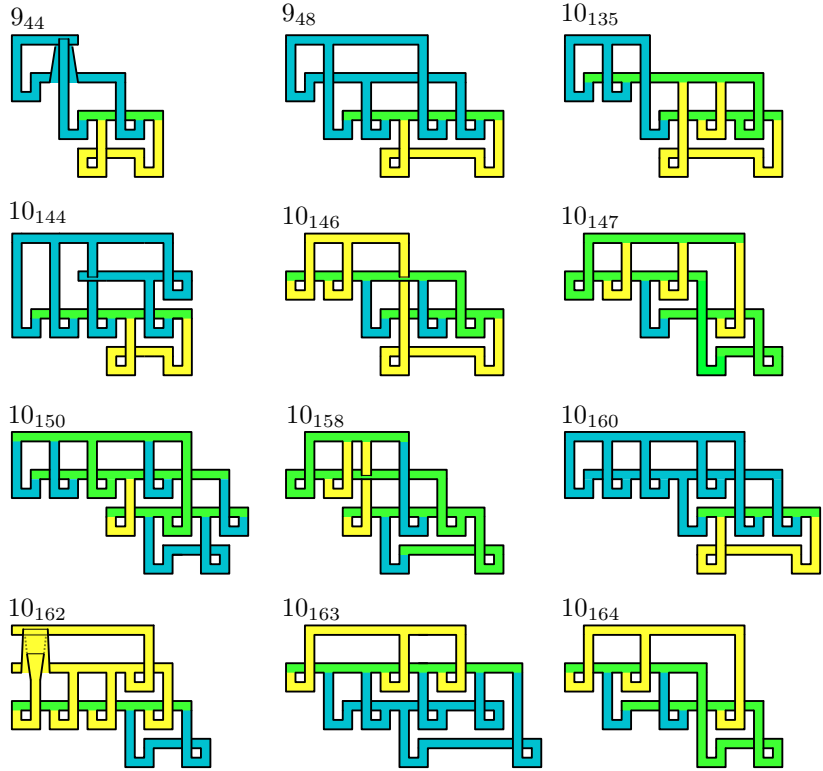
\caption{\Qh\ and \sqh\ surfaces for $9_{44}$, $9_{48}$, $10_{135}$,
  $10_{144}$,
  $10_{146}$,
  $10_{147}$,
  $10_{150}$,
  $10_{158}$,
  $10_{160}$,
  $10_{162}$,
  $10_{163}$, and
  $10_{164}$.}
\label{fig:smallexamples}
\end{figure}
\label{subsec:10-crossing}
To illustrate the concepts in more examples and towards the proof of \cref{prop:six_knots_of_ten_or_fewer_crossings}, we establish that all but 6 prime knots with 10 or fewer crossings are squeezed.
\begin{lemma}\label{lem:smallknotsaresqz}
Apart from $9_{42}$, $10_{125}$, $10_{130}$, $10_{132}$, $10_{136}$, and $10_{141}$ all prime knots with crossing number at most 10 are squeezed.
\end{lemma}
\begin{proof}
Among prime knots that have diagrams with at most 10 crossings, all but the following are slice, quasipositive, or homogeneous:
\begin{align*}
9_{42},
9_{44},
9_{48},
10_{125},
10_{130},
10_{132},
10_{135},
10_{136},
10_{141},
10_{144},
10_{146},
10_{147},
10_{150},
10_{158},\\
10_{160},
10_{162},
10_{163},
\text{ and } 10_{164}\hfill;
\end{align*}
compare~\cite{knotinfo, cromwell}. Hence by \cref{prop:all_those_classes_are_sqzd}, all prime knots that are not in the above list are squeezed.

For all the knots in the above list, except for the six knots excluded in the statement of \cref{lem:smallknotsaresqz}, we provide \qh\ surfaces or \sqh\ surfaces; hence, showing that they are all \qh\ or even \sqh. In particular, they are all squeezed. \cref{table:braids} below provides braid words given as products of elements $r$ as described in~\cref{eq:r}, such that the corresponding ribbon surface is a surface as claimed, and in \cref{fig:smallexamples}, we provide the surfaces explicitly.
\end{proof}

We end our discussion of knots with at most 10 crossings with the following observation and question.
\begin{question}
	The knots $10_{125}, 10_{130}$ and $10_{141}$ are the only knots $K$ with at most 10 crossings
	that are not quasipositive, but possess a diagram $D$ for which
	the Bennequin bound equals the Rasmussen invariant \cite{baldpla}, i.e.
	\[
	1 + \emph{writhe of } D - \emph{number of Seifert circles of }D = s(K)/2.
	\]
	Is it a coincidence that these three knots are also among the six
	knots with at most 10 crossings that are not squeezed $(10_{125})$, or not known to be squeezed
	$(10_{130}, 10_{141})$?  Are all knots with a diagram whose Bennequin bound equals the Rasmussen invariant either quasipositive, or not squeezed?
\end{question}

\subsection{Squeezing your favorite class of knots}
\label{subsec:sqzd_knots_are_fucking_everywhere}

In this subsection we further comment on \cref{prop:all_those_classes_are_sqzd}, which is ultimately verified in \cref{app:pseudo}. First we address that squeezed knots form a subgroup of the smooth concordance group.

\Cref{defn:sqzd} readily implies that, if two knots $K$ and $J$ are concordant, then $K$ is squeezed if and only if $J$ is squeezed. Furthermore, if $K$ is a squeezed knot, then so is its reverse (change the orientation of the cobordism used to squeeze~$K$) and its mirror image (consider the image of the cobordism used to squeeze $K$ under the self-diffeomorphism of $S^3\times[0,1]$ given by $(s,t)\mapsto (\Psi(s),-t)$, where $\Psi$ is an orientation-reversing self-diffeomorphism of $S^3$). In particular, $-K$ (the result of orientation reversal and mirroring) is squeezed if and only if $K$ is. So, recalling that $[-K]$ is the inverse of $[K]$ in the smooth concordance group, the only non-trivial part of verifying that
\[\mathcal{S}\coloneqq\left\{[K]\in\mathcal{C}\mid K\text{ is a squeezed knot}\right\}\]
is a subgroup of the smooth concordance group is showing that the class of squeezed knots is closed under connected sum.  We prove this as a lemma.

\begin{lemma}\label{lem:K+Jissqz}
	If $K_1$ and $K_2$ are squeezed then so is the connected sum $K_1 \# K_2$.
\end{lemma}
\begin{proof}
	Suppose that $K_i$ is squeezed by cobordisms $\Sigma_i$ between the torus knots $T^+_i$ and $T^-_i$ for each $i = 1,2$.  Then by endowing each slice of the cobordisms $\Sigma_i$ with a consistent choice of basepoint, and taking connected sum on each slice at the basepoint, we can construct a cobordism $\Sigma_1 \# \Sigma_2$ from $T^+_1 \# T^+_2$ to $T^-_1 \# T^-_2$ which has $K_1 \# K_2$ as a slice.  Now, $T^+_1 \# T^+_2$ is quasipositive and $T^-_1 \# T^-_2$ is  quasinegative by the fourth bullet point of \cref{lem:qp_facts}.  The cobordism $\Sigma_1 \# \Sigma_2$ is genus-minimizing by the third bullet point of \cref{lem:qp_facts} (noting that the genus of a quasipositive (respectively quasinegative) surface for $T_1^+ \# T_2^+$ (respectively $T_1^- \# T_2^-$) equals the sum of the genera of quasipositive (respectively quasinegative) surfaces for   $T_1^+$ and $T_2^+$ (respectively $T_1^-$ and $T_2^-$) by the second and fourth bullet points).
	
	The squeezedness of $K_1 \# K_2$ then follows from \cref{prop:equivDefSqueezed}.
\end{proof}

The next thing to note in proving the rest of \cref{prop:all_those_classes_are_sqzd} is that the statement of the proposition contains some redundancy.  We have the following diagram of proper inclusions.
See \cite{aplasq} for the inclusion of the set of almost positive knots in the set of strongly quasipositive knots,
and see \cite{kauffman} for the definition of alternative knots.
\bigskip
\begin{center}
\begin{tikzcd}[column sep=.7em]
	& \{ {\rm quasipositive} \} \arrow[r,symbol=\supset] & \{ {\rm strongly \,\, quasipositive} \} \arrow[d,symbol=\supset] \\
	\{{\rm positive \,\, torus}\} \arrow[r,symbol=\subset] & \{ {\rm positive} \} \arrow[r,symbol=\subset] \arrow[d,symbol=\subset] & \{ {\rm almost \,\, positive} \}   \\
	\{ {\rm alternative} \} \arrow[r,symbol=\subset] \arrow[d,symbol=\supset] & \{ {\rm homogeneous} \} \arrow[r,symbol=\subset]& \{ {\rm pseudoalternating} %
	\} \\
	\{ {\rm alternating} \} &&&%
\end{tikzcd}\bigskip
\end{center}

\noindent The same diagram is true when replacing all occurrences of `posi' by `nega'.  It is easy to see that quasipositive and quasinegative knots are quasihomogeneous directly from the definitions.  It remains to check that pseudoalternating knots, as defined by Mayland-Murasugi \cite{MR402718}, are quasihomogeneous. Given that the notion of pseudoalternating is not that well-known, and a large part of the proof is concerned with carefully setting up the notation, we do this in \cref{app:pseudo}. Here, instead, we directly show that alternating knots are squeezed.

\begin{proposition}
	\label{prop:alternating_knots_are_alternating_sqzd}
	Alternating knots are squeezed.
        In fact, every alternating knot is squeezed between a positive alternating knot and a negative alternating knot.
\end{proposition}

\begin{proof}
	Let $K$ be an alternating knot and let $D \subset \R^2$ be an alternating diagram for~$K$.  Let us write $o(D) \subset \R^2$ for the oriented resolution of $D$.
	We denote by $\vert o(D) \vert$ the number of components of $o(D)$ (that is, the number of Seifert circles of~$D$).
	
	We write $C^+$ (respectively $C^-$) for the set of positive (respectively negative) crossings of $D$.  We write $C = C^+ \cup C^-$ for the set of all crossings of $D$.
	
	Since $D$ is a connected diagram we can make a choice of $\vert o(D) \vert - 1$ crossings of~$D$ such that adding these crossings back to $o(D)$ gives a connected diagram.  Let us write $C^0$ for such a set of crossings.  We form the diagram $D^+$ (respectively $D^-$) by adding the crossings in $C^+ \cup C^0$ (respectively $C^- \cup C^0$) back to $o(D)$.
	
	The property of $D$ being alternating implies that each component of $\R^2 \setminus o(D)$ contains either no crossings, only positive crossings, or only negative crossings.  It follows that each negative crossing of $D^+$ is nugatory, and hence $D^+$ is a diagram of a positive link (since it admits a diagram $D^{++}$ given by changing every negative crossing of $D^+$ to a positive crossing).
	Since the canonical Seifert surface of non-split positive diagrams realizes the slice genus of those positive (in fact \sqp) links,  we have
	\[ g_4(D^{+}) = g_4(D^{++}) = (\vert C^+ \cup C^0 \vert - \vert o(D) \vert + 1) / 2 {\rm .} \]
	
	Similar reasoning gives
	\[ g_4(D^{-})  = (\vert C^- \cup C^0 \vert - \vert o(D) \vert + 1) / 2 {\rm .} \]
	
	There is a cobordism from $D^+$ to $D$ given by adding all negative crossings of $D$ not in $C^0$, then a cobordism from $D$ to $D^-$ given by removing all the positive crossings of $D$ not in $C^0$.  Each addition or removal of a crossing corresponds to the addition of a $1$-handle.  Now $D$ is a slice of the resulting cobordism $\Sigma$ from $D^+$ to~$D^-$.  If we can show that $\Sigma$ minimizes the genus among cobordisms between $D^+$ and $D^-$ we will have shown that $D$ is squeezed, using \cref{prop:equivDefSqueezed}.
	
	The number of $1$-handles of $\Sigma$ is
	\[ \vert C^- \setminus C^0 \vert + \vert C^+ \setminus C^0 \vert = \vert C \setminus C^0 \vert = \vert C \vert  - \vert C^0 \vert = \vert C \vert - \vert o(D) \vert + 1 {\rm .} \]
	So $\Sigma$ has genus $g(\Sigma) = (\vert C  \vert - \vert o(D) \vert + 1) / 2$.  On the other hand, the sum of the slice genera of $D^+$ and $D^-$ is
	\begin{align*}
	(\vert C^+ \cup C^0 \vert - \vert o(D) \vert + 1) / 2 &+(\vert C^- \cup C^0 \vert - \vert o(D) \vert + 1) / 2 \\
	& = (\vert C \vert + \vert C^0 \vert - 2 \vert o(D) \vert + 2) /2 \\
	& = (\vert C  \vert - \vert o(D) \vert + 1) / 2 = g(\Sigma) { \rm ,}
	\end{align*}
	so we are done.

To see the second part of the proposition, we
note that if the diagram $D^+$ is a diagram of a link with more than one component, then it can be turned into a diagram $D^{+++}$ of a knot by adding positive crossings between Seifert circles that are not connected by a crossing in $C_0$. Note that $D^{+++}$ is again an alternating diagram with the same number of Seifert circles $o(D)$. Similarly, we find a diagram $D^{---}$ of a knot by adding negative crossings to $D^{-}$. Extending the cobordism $\Sigma$ to a cobordism between the knots defined by $D^{+++}$ and $D^{---}$, respectively, by adding $1$-handles corresponding to the new crossings, we see that $K$ is a slice of a genus-minimizing genus cobordism between the knots given by the diagrams $D^{+++}$ and~$D^{---}$.
\end{proof}

\begin{remark}\label{rem:altaresq}
In the above proof of \cref{prop:alternating_knots_are_alternating_sqzd} we avoided the notion of \sqh\, to have a rather self-contained proof. However, we note that one easily explicitly exhibits a Seifert surface for an alternating knot $K$ that is \sqh\ as follows. Seifert's algorithm applied to $D$ yields a Seifert surface $F$ that is the union of Seifert surfaces $F_+$ and $F_-$ obtained by Seifert's algorithm from the diagrams $D^+$ and $D^-$, respectively, and $F_+\cap F_-$ is the closed disc obtained by Seifert's algorithm applied to the diagram $D^0$ given by adding $C^0$ back to $o(D)$. Since $F_+$ and $F_-$ are isotopic to surfaces obtained by Seifert's algorithm applied to the positive alternating diagram $D^{++}$ and the negative alternating diagram $D^{--}$, respectively, $F_+$ and $F_-$ are \sqp\ and \sqn\ surfaces, respectively. Hence $F$ is \sqh\ and thus $K$ is \sqh.
\end{remark}

We end the section with a question related to the inclusions
\begin{center}
$\displaystyle
\text{\{positive\}}\subset\text{\{almost positive\}}\subset\text{\{strongly quasipositive\}}\subset\text{\{squeezed\}}.$
\end{center}
\begin{question}
        If a knot $K$ admits a diagram with at most one negative crossing, then $K$ is squeezed.
	On the other hand, $10_{132}$ admits a diagram with three negative crossings and is not squeezed.
	Is every knot admitting a diagram with two negative crossings squeezed?
\end{question}

\section{Obstructions}
\label{sect:obs}
In this section we verify \cref{prop:boring_on_squeezed} and make precise how $s$ determines the invariants from the five bullet points of the proposition. But before we begin, we first show how \cref{prop:boring_on_squeezed} may be used to verify that the four knots of \cref{prop:six_knots_of_ten_or_fewer_crossings} are not squeezed.

\begin{example}\label{ex:notsqzd}
	The knots $9_{42}$ and $10_{136}$ have $s^{\text{Sq}^2}_+ = 2 \neq 0 = s$,
	and the knot $10_{132}$ has $s^{\text{Sq}^2}_+ = 0 \neq -2 = s$ \cite{LipSar3}.
	Here, $s^{\text{Sq}^2}_+$ is a refinement of $s$ defined by Lipshitz-Sarkar using the stable cohomology operation called the second Steenrod square. We will make the second bullet point of \cref{prop:boring_on_squeezed} precise by showing $s^{\text{Sq}^2}_+(K)=s(K)$ for squeezed knots; see \Cref{lem:ls}.
	Hence those three knots are not squeezed.
        Alternatively, their squeezedness may be obstructed by Sarkar-Scaduto-Stoffregens's refinements of $s$ coming from the first Steenrod square operating on odd Khovanov homology \cite{MR4078823} (see \cite{schuetzodd} for the computation of these refinements on $9_{42}$, $10_{132}$ and $10_{136}$).
	
	Denote by $s_3$ the $\mathfrak{sl}_3$-Rasmussen invariant, normalized to be a slice torus invariant (see \cref{defn:slice_torus} below).
	The knot $10_{125}$ has $s_3 = -1/2$, but $s = -2$ \cite{lew2}. Since $s/2$ is another slice torus invariant,
	and $-1/2\neq -1$, we see that the knot $10_{125}$ is also not squeezed.
	Non-squeezedness of $10_{125}$ may also be established by observing that $S_3$ of $10_{125}$ is of rank 3,
	or that $\gimel_3$ of $10_{125}$ is not linear (see \cref{subsec:khov_roz_stuff} for details on the Khovanov-Rozansky squeezedness obstructions $S_n$ and $\gimel_n$).
\end{example}

We think that the investigation of what knot invariants can be used to obstruct squeezedness, can be understood as a test to the strength of the concordance invariants developed so far. Concretely, we wonder whether the technology has progressed to the point that the following can be answered in the positive.
\begin{question}
Does $\mathcal{C}/\mathcal{S}$ -- the quotient of the smooth concordance group by the subgroup of squeezed knots -- contain a free Abelian subgroup of infinite rank? Does it contain a free Abelian summand of infinite rank?
\end{question}

To construct such a subgroup, one would need infinite families of non-squeezed knots. Let us mention a potential example of such a family.

\begin{question}
        The $n$-twisted positive Whitehead double $D_+(K,n)$ of a knot $K$ is strongly quasipositive if $n$
        is less than or equal to the Thurston-Bennequin number $TB(K)$, and strongly quasihomogeneous if
        $n \geq -TB(-K)$ (cf.~\cite[Theorem~2]{doubled}).
        Is $D_+(K,n)$ non-squeezed if $TB(K) < n < -TB(-K)$?
        In some cases this can be shown using differing slice-torus invariants.
        But e.g.~we do not know whether the knots $D_+(T(2,3),4)$ and $D_+(T(2,3),5)$ are squeezed.
\end{question}

We start in \cref{subsec:slicetorus} with the first bullet point of \cref{prop:boring_on_squeezed} and verify that all slice torus invariants agree on squeezed knots.  We then use this in \cref{subsec:sqzing_framework} as motivation to give a framework that applies to the remaining bullet points of \cref{prop:boring_on_squeezed}, which we which consider in \cref{subsec:lipsarinvrt,subsec:sss,subsec:sch,subsec:khov_roz_stuff}.
\subsection{Slice torus invariants}
\label{subsec:slicetorus}
The following definition is due to Livingston \cite{livingston} (see also \cite{lew2}).
\begin{definition}
	\label{defn:slice_torus}
	Writing $\cC$ for the smooth concordance group of knots, a \emph{slice torus invariant} is a homomorphism $\phi \colon \cC \rightarrow \R$ satisfying two conditions:\\[1ex]
	\begin{tabular}{@{\hspace{3em}}l@{\ }lr@{\ }ll}
		 & \textsc{Slice:} & $\phi(K)$ & $\leq g_4(K)$ & for all knots $K$ and \\
		 & \textsc{Torus:} & $\phi(T_{p,q})$ & $= g_4(T_{p,q})$ & for all positive coprime integers $p,q$.
	\end{tabular}
\end{definition}

The known slice torus invariants are (after proper normalization) the $\tau$ invariant from Heegaard-Floer homology \cite{osz10},
and the Rasmussen invariant $s$ from Khovanov homology \cite{ras3}, together with its variations
coming from Khovanov homology over $\mathbb{F}_p$ for any prime $p$,
and from $\mathfrak{sl}_n$-homology for any $n\geq 2$, where one may choose different Frobenius algebras \cite{wu3,lobb1,lobb2,lewarklobb}.

We prove the first bullet point of \cref{prop:all_those_classes_are_sqzd}.

\begin{lemma}
	\label{lem:slicetorus_are_obstructions}
	If $\phi_1$ and $\phi_2$ are slice torus invariants and $K$ is a squeezed knot then~$\phi_1(K) = \phi_2(K)$.
\end{lemma}

\begin{proof}
	For any two knots $J,L$ we have that the minimal genus of a cobordism between $J$ and $L$ is equal to $g_4(J \# -\!L)$.  Let $i \in \{ 1 , 2 \}$.  Since $\phi_i$ is a homomorphism we also have
	\[ \phi_i(J) - \phi_i(L) = \phi_i(J \# -\!L) \leq g_4(J \# -\!L) {\rm .} \]
	
	Now suppose that $K$ is squeezed between the torus knots $T_{p,q}$ and $-T_{r,s}$ where $p,q,r,s \geq 1$.   This means that there is a genus-minimizing cobordism $\Sigma$ from $T_{p,q}$ to $-T_{r,s}$ that decomposes as the composition of a cobordism $\Sigma_1$ from $T_{p,q}$ to $K$ and a cobordism $\Sigma_2$ from $K$ to $-T_{r,s}$.  Since $\Sigma$ is genus-minimizing, so are both $\Sigma_1$ and~$\Sigma_2$.
	
	We have
	\begin{equation}\label{eq:dagger} \phi_i(T_{p,q}) - \phi_i(K) \leq g(\Sigma_1) \,\,\, {\rm and} \,\,\, \phi_i(K) - \phi_i(-T_{r,s}) \leq g(\Sigma_2) \end{equation}
	but also
	\[ (\phi_i(T_{p,q}) - \phi_i(K)) + (\phi_i(K) - \phi_i(-T_{r,s})) = \phi_i(T_{p,q}) - \phi_i(-T_{r,s}) = g(\Sigma) = g(\Sigma_1) + g(\Sigma_2) {\rm .} \]
	
	Hence we must have had equalities in \cref{eq:dagger}, so we see (recalling that since $\phi_i$ is a slice torus invariant we have $\phi_i(T_{p,q}) = g_4(T_{p,q})$ and $\phi_i(-T_{r,s}) = -g_4(T_{r,s})$)
	\[ \phi_1(K) = \phi_2(K) = g_4(T_{p,q}) - g_4(T_{p,q} \# -\!K) = g_4(K \# -\!T_{r,s}) - g_4(T_{r,s}){\rm .}\myqed \]
\end{proof}
Below we explain in detail the squeezing obstructions from \cref{prop:all_those_classes_are_sqzd} that go beyond slice torus invariants, which in particular we used to establish that $9_{42}$ is not squeezed in \cref{ex:notsqzd}.
However, we wonder whether it is actually necessary to go beyond slice torus invariants to obstruct squeezedness.
\begin{question}
	Do there exist slice torus invariants $\phi_1$ and $\phi_2$ for which we have $\phi_1(9_{42}) \not= \phi_2(9_{42})$?  More generally, given a non-squeezed knot $K$ can one always find slice torus invariants that differ on $K$?
\end{question}

\subsection{A squeezing framework}
\label{subsec:sqzing_framework}
We use the proof of \cref{lem:slicetorus_are_obstructions} to guide us towards proving the rest of \cref{prop:boring_on_squeezed}.

We interpret the invariants of \cref{prop:boring_on_squeezed} as maps from $\cC$ into metric spaces $M$ in which the distance between points gives a lower bound on cobordism distance.  In the case of slice torus invariants the metric space is $M = \R$ with its usual metric.  We are interested in those metric spaces which contain an isometric copy of $\Z \subseteq M$, and we show that the invariants of interest map any squeezed knot $K$ to $s(K)/2 \in \Z \subseteq M$.  Let us now be more precise.

\begin{definition}
	\label{def:general_sqzd_obstruction}
Let $(M,d)$ be a metric space
containing an isometric copy of $\Z \subseteq M$ (where $\Z$ carries the standard metric $d(n,m) = |n-m|$ for $n,m\in\Z$)
such that for all $m,n \in \Z$ with $m \geq n$ we have
\begin{equation}\label{eq:frameworkcondition}
d(m,x)\in\Z_{\geq 0} \text{ and } d(m,x) + d(x,n)  = d(m,n)
\; \Longrightarrow\; x \in \{n, n + 1, \ldots, m\} {\rm .}
\end{equation}
Let $\mathfrak{T}_+, \mathfrak{T}_- \subseteq \cC$ be the sets of positive and negative torus knots, respectively.
We call a function $y\colon\mathcal{C}\to M$ a \emph{squeezing obstruction} if and only if 
for all knots $K, J$ we have
\begin{equation}\label{eq:frameworkcondition2}
	d(y(K),y(J)) \leq g_4(K \# -\!J),
\end{equation}
and for all knots $T_{\pm} \in \mathfrak{T}_{\pm}$ we have
\begin{equation}\label{eq:frameworkcondition3}
	y(T_{\pm})  = \pm g_4(T_{\pm}) {\rm .} 
\end{equation}
\end{definition}

\begin{lemma}\label{lem:framework}
	Let $K$ be squeezed between the knot $T_+ \in \mathfrak{T}_+$ and the knot $T_- \in \mathfrak{T}_-$,
	then for all squeezing obstructions $y\colon\mathcal{C}\to M$ we have 
	\[
	y(K) = g_4(T_+) - g_4(T_+ \# -\!K) = g_4(K \# -\!T_-) - g_4(T_-) {\rm .}
	\]
\end{lemma}
\begin{proof}
	The triangle inequality implies
	\[
	d(y(T_+), y(K)) + d(y(K), y(T_-)) \ \geq\ d(y(T_+),y(T_-)).
	\]
	But we also have
	\begin{align*}
		d(y(T_+), y(K)) + d(y(K), y(T_-)) & \leq g_4(T_+ \# -\!K) + g_4(K \# -\!T_-) \\
		& = g_4(T_+ \# -\!T_-) \\
		& = g_4(T_+) + g_4(T_-) \\
		& = d(g_4(T_+),-g_4(T_-)) \\
		& = d(y(T_+), y(T_-)).
	\end{align*}
	This implies
	\[
	d(y(T_+), y(K)) + d(y(K), y(T_-)) \ =\ d(y(T_+),y(T_-),
	\]
        which moreover implies that the inequality $d(y(T_+), y(K)) \leq g_4(T_+ \# -\!K)$ is an equality, and so $d(y(T_+), y(K))\in \Z_{\geq 0}$. The claimed formula for $y(K)$ now follows from equation \cref{eq:frameworkcondition}.
\end{proof}

\subsection{The Lipshitz-Sarkar refinements of the Rasmussen invariant}
\label{subsec:lipsarinvrt}
Lipshitz and Sarkar defined refinements of the Rasmussen invariant \cite{LipSar3},
which, in a nutshell, arise from non-trivial stable cohomology operations involving Lee generators in spacified Khovanov homology.

Denote by $\ls\colon \mathcal{C}\to 2\Z$ any of their invariants $r^{\alpha}_+, r^{\alpha}_-, s^{\alpha}_+, s^{\alpha}_-$ for any cohomology operation $\alpha$, over any ground field $\mathbb{F}$.
Lipshitz and Sarkar prove that
\[ \vert \ls(K) - s(K) \vert \in \{0,2\} \,\,\, {\rm and} \,\,\, |\ls(K) - \ls(J)| \leq 2g_4(K \# -\!J) \]
for all knots $K, J$.
We verify the following lemma which establishes the second bullet point of \cref{prop:boring_on_squeezed}.
\begin{lemma}\label{lem:ls}
If $\ls$ is any Lipshitz-Sarkar refinement, and $K$ a squeezed knot,
then we have that $\ls(K) = s(K)$.
\end{lemma}

To appreciate the proofs of the following two propositions, some knowledge about Khovanov homology is required.
As far as notation is concerned, we write $CKh^{t,q}$ and $Kh^{t,q}$ for the Khovanov chain complex and its homology, respectively, in homological grading $t$ and quantum grading $q$; we write $v_{\pm}$ for the usual generators of
the underlying Frobenius algebra.

\begin{proof}[Proof of \cref{lem:ls}]
	We claim $\frac{1}{2}\ls\colon\mathcal{C}\to\Z$ is a squeezing obstruction (see \cref{def:general_sqzd_obstruction}) so the result follows from \cref{lem:framework}.
	
	As metric space we just take $\Z$.  While $|\ls(K) - \ls(J)| \leq 2g_4(K \# -\!J)$ has been established by Lipshitz-Sarkar, the values of $\ls$ on torus knots do not seem to have been computed before.
	By \cite[Lemma~5.1]{LipSar3}, it is enough to show that $Kh^{t,q}$ is trivial if $t \neq 0$ and $q = s(K) \pm 1$
	(because the relevant cohomology operations are homomorphisms $Kh^{-n,q} \to Kh^{0,q}$
	or $Kh^{0,q} \to Kh^{n,q}$ for $q = s(K) \pm 1$ and some $n\geq 1$).
	This is established in \cref{lem:posbraidkh}.
\end{proof}
In fact, we prove the following not just for torus knots, but for positive braid links.
Compare our statement to Sto\v{s}i\'{c}'s theorem that $Kh^{1,*}(L)$ is trivial for positive braid links $L$ \cite{stosic}.
\begin{lemma}\label{lem:posbraidkh}
	Let $L$ be a link that arises as the closure of a positive braid word $b$ with $n$ strands and $c$ crossings.
	Then $Kh^{t,q}(L)$ is trivial if $t \neq 0$ and $q = 1 + c - n \pm 1$.
\end{lemma}
\begin{proof}
	Let $D$ be the diagram for $L$ that arises as closure of $b$.
	Because all crossings of $D$ are positive, $CKh^{t,q}$ is trivial if $t \leq -1$.
	Since $D$ is the closure of a positive braid, it is A-adequate,
	meaning that all crossings in the unique resolution at homological grading $t = 0$ are between two different circles.
	So all resolutions at $t = 1$ have $n-1$ circles, i.e.\ one circle less than the unique resolution at $t=0$.
	Therefore, $CKh^{t,q}$ is trivial if $t \geq 1$ and $q = c - n$.
	It remains to show that $CKh^{t,\tilde q}$ is trivial for $t \geq 1$ and $\tilde q = 2 + c - n$.
	
	For this purpose, it is enough to show that the homology of the (ungraded) chain complex $CKh^{*,\tilde q}$
	consists only of one copy of $\Z$ (which lives at $t = 0$).
	At $t = 0$, the unique resolution consists of $n$ circles, one for each strand of the braid~$b$.
	We enumerate those circles in the same way as the strands of $b$, so that a crossing corresponding to a braid generator $a_i$ touches the $i$-th and $(i+1)$-st circle. The group $CKh^{0,\tilde q}$ is generated
	by $w_1, \ldots, w_n$, where $w_i$ denotes the tensor product of 
	$n-1$ factors $v_-$ and one factor $v_+$ at index~$i$.
	We will need the following basis $(u_i)_{i=1}^n$ for $CKh^{0,\tilde q}$:
	\[
	u_1 = w_1,\quad
	u_2 = w_2 - u_1,\quad
	\ldots,\quad
	u_n = w_n - u_{n-1},
	\]
	These generators were chosen because for $1\leq i\leq n-1$,
	$d(u_i)$ is supported only in resolutions that $1$-resolve only crossings corresponding to a generator $a_i$ of $b$,
	and $d(u_n) = 0$.
	
	At $t \geq 1$, $CKh^{t,\tilde q}$ is supported in resolutions consisting of $n + t - 2$ circles,
	where it is generated by tensor products of only $v_-$ factors.
	For $1\leq i\leq n-1$, let $X_i$ be the subgroup of $CKh^{*,\tilde q}$ generated by $u_i$ and for $t\geq 1$
	by those resolutions for which only crossings corresponding to a generator $a_i$ of $b$
	are 1-resolved. One checks that $X_i$ is a subcomplex, and the chain complex $CKh^{*,\tilde q}$ decomposes as
	\[
	CKh^{*,\tilde q} = \langle u_n\rangle \oplus \bigoplus_{i=1}^n X_i.
	\]
	So, to establish $H(CKh^{*,\tilde q}) \cong \Z$, it suffices to show that $X_i$ has trivial homology,
	which we will do now.
	For $k \geq 0$, let $D(k)$ be the diagram of the $T(2,k)$ torus link arising as closure of the braid $a_1^k$.
        Let $c_i \geq 0$ be the number of occurrences of $a_i$ in~$b$.
        Then the chain complex $CKh^{*,c_i}(D(c_i))$ decomposes as $\langle r_i\rangle \oplus Y_i$,
	where $r_i$ is a cycle in homological grading~0, and $Y_i$ is isomorphic to $X_i$.
	The Khovanov homology of $T(2,k)$ torus links is well-known; in particular, the homology of $CKh^{*,c_i}(D(c_i))$
	is~$\Z$. It follows that $Y_i$ has trivial homology, and thus so does $X_i$.
	This concludes the proof.
\end{proof}
\subsection{The Sarkar-Scaduto-Stoffregen refinements of Rasmussen's invariant}\label{subsec:sss}
Building on Lipshitz-Sarkar's work discussed in the last subsection,
Sarkar-Scaduto-Stoffregen \cite{MR4078823} constructed a spacification of
the variation of Khovanov homology introduced by {Ozsv\'{a}th-Rasmussen-Szab\'{o} \cite{MR3071132},
called odd Khovanov homology, and extracted concordance invariants from it.
The following lemma proves the third bullet point of \cref{prop:boring_on_squeezed}.
\begin{lemma}\label{lem:sss}
If $\sss$ is any of the Sarkar-Scaduto-Stoffregen invariants
$r^{\alpha}_+, r^{\alpha}_-, s^{\alpha}_+, s^{\alpha}_-$ for any cohomology operation $\alpha$, over the ground field $\mathbb{F}_2$,
and $K$ a squeezed knot, then we have that $\sss(K) = s(K)$.
\end{lemma}
\begin{proof}
The proof proceeds in the same way as the proof of \cref{lem:ls}.
Sarkar-Scaduto-Stoffregen prove that
\[ |\sss(K) - \sss(J)| \leq 2g_4(K \# -\!J) \]
for all knots $K, J$. So, to show that $\tfrac12\sss$ is a squeezing obstruction, we just need to compute that $\sss(T) = s(T)$ for all torus knots $T$. Since Khovanov homology and odd Khovanov homology are isomorphic with $\mathbb{F}_2$ coefficients, this equality follows from \cref{lem:posbraidkh} by the same argumentation as in the proof of \cref{lem:ls}.
\end{proof}
\subsection{\texorpdfstring{Sch\"utz's integral Rasmussen invariant $s^{\mathbb{Z}}$}{Sch\"utz's integral Rasmussen invariant s\^{}Z}}
\label{subsec:sch}
For every field $\mathbb{F}$, there is a Rasmussen invariant $s^{\mathbb{F}}$,
which is a slice torus invariant. The original Rasmussen invariant $s$ is equal to $s^{\mathbb{Q}}$. Sch\"utz recently defined an integral version of the Rasmussen invariants \cite{schuetz}.
It consists of an invariant $\gl(K)$, which is a non-negative integer, and 
an invariant $s^{\Z}(K) \in (2\Z) \times \Z^{\gl(K)}$.
The first entry of $s^{\Z}(K)$, i.e.\ the $2\Z$ factor, always equals~$s(K)$.
The following lemma proves the fourth bullet point of \cref{prop:boring_on_squeezed}.%
\begin{lemma}
If $K$ is a squeezed knot, then $\gl(K) = 0$ and thus $s^{\Z}(K) = s(K)$.
\end{lemma}
\begin{proof}
The set $\Z^2$ with the maximum metric $d_{\max}$, and the diagonal $\{(n,n)\mid n\in\Z\}$ as isometric copy of $\Z$, clearly satisfies \cref{eq:frameworkcondition}. We will prove that
\[\ds\colon K\mapsto \begin{pmatrix} s(K)/2\\ s(K)/2 - \gl(K)\end{pmatrix}\]
is a squeezing obstruction in the sense of \cref{def:general_sqzd_obstruction}.
Then, the lemma follows from \cref{lem:framework}.

To show that $\ds$ is a squeezing obstruction, we need to check for all torus knots~$T$ that $\ds(T) = (s(T)/2, s(T)/2)$, i.e.~that $\gl(T) = 0$; and for all knots $K$, $J$ that $g_4(K\# \!-J) \geq d_{\max}(\ds(K), \ds(J))$.
Let us start by tackling the second statement. Since we know that $g_4(K\#\! -J) \geq |s(K)/2 - s(J)/2|$,
it suffices to prove
\begin{equation}\label{eq:schuetz}
g_4(K\#-\!J) \geq \Bigl|\frac{s(K)}2 - \gl(K) - \frac{s(J)}2 + \gl(J)\Bigr|.
\end{equation}
Sch\"utz shows that for all knots $L$,
$g_4(L) \geq |s(L)/2 - \gl(L)|$.
Moreover, he proves that for all knots $K_1, K_2$, we have the inequality $\gl(K_1 \# K_2) \leq \gl(K_1) + \gl(K_2)$.
Because $\gl$ is invariant under concordance, it follows that $\gl(K) \leq \gl(K\#-\!J) + \gl(J)$.
Combining these statements yields
\begin{align*}
g_4(K\#-\!J) & \geq \Bigl|\frac{s(K)}2 - \frac{s(J)}2 - \gl(K\#-\!J)\Bigr| \\
             & \geq \frac{s(J)}2 - \frac{s(K)}2 + \gl(K\#-\!J) \\
             & \geq \frac{s(J)}2 - \frac{s(K)}2 + \gl(K) - \gl(J).
\end{align*}
Switching the roles of $K$ and $J$, we find
\[
g_4(K\#-\!J) = g_4(J\#-\!K) \geq \frac{s(K)}2 - \frac{s(J)}2 + \gl(J) - \gl(K).
\]
All in all, we have proven \cref{eq:schuetz}.

It remains to show $\gl(T) = 0$ for torus knots $T$. If $T$ is a negative torus knot, then we have
\[
-s(T)/2 = g_4(T) \geq |s(T)/2 - \gl(T)| \geq -s(T)/2 + \gl(T),
\]
which implies $0 \geq \gl(T)$ and thus $\gl(T) = 0$. If $T$ is a positive torus knot (or any positive knot, in fact), pick a positive diagram $D$ for $T$. Then the Khovanov chain complex $C^*_{\text{BN}}(D; \Z)$, which Sch\"utz uses to define $s^{\Z}(T)$, is supported in non-negative homological degrees. It follows that there is no torsion in homological degree 0 on any page of the spectral sequence of $C^*_{\text{BN}}(D; \Z)$. Consequently, $\gl(T) = 0$.
\end{proof}
\subsection{The obstructions from Khovanov-Rozansky homology}
\label{subsec:khov_roz_stuff}
The second and third author defined a refinement $S_n\colon \mathcal{C}\to\mathcal{G}_n$ of the $\mathfrak{sl}_n$-Rasmussen invariant \cite{gimel}. Here $\mathcal{G}_n$ is the set
of isomorphism classes of finitely generated indecomposable chain complexes $C$ of free graded $R_n = \C[x, a_1, \ldots, a_{n-1}]$-modules (where the grading is $\deg x = 2, \deg a_i = 2(n - i)$),
such that
\[
H_*(S_n(K) \otimes E) = H_0(S_n(K) \otimes E) \cong \C,
\]
where $E$ is the $R_n$-module $R_n / (a_1, \ldots, a_{n-1}, x-1) \cong \C$.
For $n = 2$, $S_2(K)$ is of rank one, in quantum grading $s(K)$; thus $S_2$ is fully determined by the Rasmussen invariant.
For $n \geq 3$, there are knots $K$ such that $S_n(K)$ is of higher rank.
Furthermore, there is a smooth concordance invariant $\gimel_n$, which associates to a concordance class
a piecewise linear function $[0,1]\to\R$.  This invariant $\gimel_n(K)$ depends only on~$S_n(K)$.

We prove the following lemma which establishes the fifth and final bullet point of \cref{prop:boring_on_squeezed}.
\begin{lemma}\label{lem:Sn}
For all $n\geq 2$ and all squeezed knots $K$, $S_n(K)$ is of rank one, in quantum grading $(n-1)s(K)$;
and $\gimel_n(K)$ is linear with slope $\frac{1}{2} s(K)$.
In particular, $S_n(K)$ (and hence $\gimel_n(K)$ as well) is fully determined by~$s(K)$.
\end{lemma}

The following proof of the obstruction coming from the $S_n$-invariant uses the language of \cite[Section~5]{gimel}.
\begin{proof}[Proof of \cref{lem:Sn}]
Let us define a metric $d$ on $\mathcal{G}_n$ as follows. For $C_1, C_2 \in \mathcal{G}_n$,
$d(C_1, C_2)$ is the minimum non-negative integer
for which there are graded chain maps $f_1\colon C_1 \to C_2$ and $f_2\colon C_2\to C_1$ of degree $2(n-1)d(C_1, C_2)$,
such that the maps
\[
\C \cong H(C_1 \otimes E) \to H(C_2 \otimes E) \cong \C \,\,\, {\rm and} \,\,\,
\C \cong H(C_2 \otimes E) \to H(C_1 \otimes E) \cong \C
\]
induced by $f_1$ and $f_2$ respectively are both non-zero.  

Let us show that $d$ is indeed a metric.
The symmetry of $d$ is immediate, and the triangle inequality easily follows since $E$ is a domain (in fact, a field).
Now assume $d(C_1, C_2) = 0$.
Then $f_2\circ f_1$ is an endomorphism of $C_1$ of degree 0.
The category of finitely generated graded chain complexes over $R_n$ has the property that every endomorphism of an indecomposable is either an isomorphism or nilpotent~\cite{gimel}.
This applies to the map $f_2\circ f_1$.
It cannot be nilpotent, since the induced endomorphism of $H(C_1 \otimes E)$ is not.
It follows that $f_2 \circ f_1$, and similarly $f_1\circ f_2$, is an isomorphism. Since the elements of $\mathcal{G}_n$ are indecomposable, it follows that $f_1$ and $f_2$ are isomorphisms as well.

The integers $\Z$ are included in $\mathcal{G}_n$ by taking $m \in \Z$ to be the chain complex consisting of one copy of $R_n$ in homological grading $0$ and quantum grading~$2(n-1)m$.  For notational clarity, we write this chain complex as $\underline{m}$.  Let us show that this inclusion $\Z \subseteq \mathcal{G}_n$ is an isometric embedding.

For $k\leq m$ consider a graded map $f\colon \underline{k}\to \underline{m}$
such that $f_*\colon H(\underline{k}\otimes E)\to H(\underline{m}\otimes E)$ is non-zero.
It follows that $f(1) \neq 0$. Since the grading of $1\in\underline{k}$ is $2(n-1)k$,
and the non-zero element of minimal quantum grading in $\underline{m}$ is $1$ with grading $2(n-1)m$,
it follows that the degree of $f$ is at least $2(n-1)(m-k)$. Thus $d(\underline{k}, \underline{m}) \geq m-k$.
On the other hand, $d(\underline{k}, \underline{m}) \leq m-k$ may be shown by observing that the maps
\begin{align*}
	f\colon \underline{k}\to \underline{m},&\ f(1) = 1,\\
	g\colon \underline{m}\to \underline{k},&\ f(1) = x^{(n-1)(m-k)}
\end{align*}
are graded of degree $m-k$ and induce non-zero maps on $H(\,\cdot\,\otimes E)$.

Now we prove that $S_n$ is a squeezing obstruction in the sense of \cref{def:general_sqzd_obstruction}, thus establishing the result required by \cref{lem:framework}.
To verify condition \cref{eq:frameworkcondition} in \cref{def:general_sqzd_obstruction}, assume that for some chain complex $C$ and integers $\ell, k, m$ with $\ell\geq 0$ and $k\leq m$ we have
\[
d(\underline{k}, C) = \ell \,\,\, {\rm and} \,\,\, d(C, \underline{m}) = m - k - \ell {\rm .}
\]
This means that we have maps $f\colon \underline{k}\to C$ and $g\colon C\to \underline{m}$
of degree $2(n-1)\ell$ and $2(n-1)(m-k-\ell)$, respectively, that induce non-zero maps on $H(\,\cdot\,\otimes E)$.
Denote by $\{\,\cdot\,\}$ the quantum grading shift operator.
By shifting, we obtain maps $f'\colon \underline{k} \to C\{-2(n-1)\ell\}$ and
$g'\colon C\{-2(n-1)\ell\} \to \underline{m}\{-m+k\}$, both of degree~$0$.
Note that $\underline{m}\{-m+k\} \cong \underline{k}$. So, by similar reasoning as in the above proof of the identity of indiscernibles for $d$,
it follows that $f'$ and $g'$ are isomorphisms, and thus $C\cong \underline{k + \ell}$.

To verify \cref{eq:frameworkcondition2}, note that a smooth cobordism $F$ from $K$ to $J$ induces
a graded chain map $F_*\colon S_n(K) \to S_n(J)$ of degree $2(n-1)g(F)$, such that the map
$H(S_n(K) \otimes E) \to H(S_n(J) \otimes E)$ is non-zero \cite{gimel}.
Thus we have
\[ d(C_1, C_2) \leq g(F) \leq g_4(K\# -\! J) {\rm .} \]

Finally, we verify \cref{eq:frameworkcondition3}.
To compute the $S_n$-invariant of torus knots, consider a knot diagram $D$
of a positive torus knot $K$, such that $D$ is the closure of a positive braid.
One checks that the graded $R_n$-chain complex $C_{U(n)}(D)$ has a free summand of rank $1$
living in homological grading $0$, at quantum grading $2(n-1)g_4(K)$. It follows
that $S_n(K) = \underline{g_4(K)}$. Moreover, $S_n(-K)$ is the dual of $S_n(K)$,
which is isomorphic to $\underline{-g_4(K)}$.
\end{proof}

\section{Some avenues to explore and some to avoid}
\label{sect:extra}
In this section we briefly pursue some natural questions that arise from our study of squeezedness.  In \cref{subsec:more_on_obstructions}, we consider some situations in which we are more restrictive on those knots between which we squeeze.  In \cref{subsec:non-obstructions} we consider some invariants that the reader may be surprised to learn do not provide obstructions to being squeezed.

\subsection{Obstructions concerning squeezing between particular kinds of knots and links}
\label{subsec:more_on_obstructions}
In \cref{rmk:g2}, we have seen that the slice genus of a squeezed knot is not determined by its Rasmussen invariant.
However, we can prove that under a stronger hypothesis, the slice genus is indeed determined by the Rasmussen invariant.

\begin{proposition}\label{prop:possqueezed}
	A knot $K$ is squeezed between a positive torus knot and the unknot if and only if it is squeezed and for one (and thus for every) slice torus invariant $\phi$ we have
	\[
	\phi(K) = g_4(K).
	\]
\end{proposition}

\begin{proof}[Proof of \Cref{prop:possqueezed}]
	To show the `if' direction, assume that $\phi(K) = g_4(K)$ and that $K$ is squeezed,
	i.e.~there is a genus-minimizing cobordism $\Sigma$ between a positive torus knot $T^+$ and a negative torus knot $T^-$,
	such that $K$ appears as a slice of $\Sigma$. Let $\Sigma^+$ be the part of $\Sigma$ between $K$ and $T^+$.
	We have
	\[
	g(\Sigma^+) = g_4(T^+ \# -\!K) = g_4(T^+) - \phi(K) = g_4(T^+) - g_4(K),
	\]
	where the second equality was shown in \cref{lem:framework}.
	Let $\Sigma'$ be the composition of $\Sigma^+$ with a genus-minimizing slice surface of $K$.
	Then
	\[
	g(\Sigma') = g(\Sigma^+) + g_4(K) = g_4(T^+).
	\]
	Therefore, $\Sigma'$ is a genus-minimizing slice surface of $T^+$, and thus a positive squeezing cobordism for $K$.
	
	To show the `only if' direction, we rely once again on \cref{lem:framework}.  Note that for any slice torus invariant $\phi$ (which is a squeezedness obstruction) we have
	\[ \phi(K) = g_4(K \# U) - g_4(U) = g_4(K) \]
	where $U$ is the unknot.
\end{proof}
Let $\nu_+$ and $\varepsilon$ be the Heegaard-Floer knot concordance invariants introduced in \cite{MR3523259}
and \cite{MR3217622,MR3260841}, respectively.
\begin{corollary}\label{cor:epsilonforpossqueezed}
	If a knot $K$ is squeezed between a positive torus knot and the unknot, then $\nu_+(K) = g_4(K)$ and $\varepsilon(K) = \sgn g_4(K)$.
\end{corollary}
\begin{proof}
	By \Cref{prop:possqueezed} , $g_4(K) = \tau(K)$.
	We have $\tau(K) \leq \nu_+(K) \leq g_4(K)$ \cite{MR3523259}, so $\nu_+(K) = g_4(K)$ follows.
	Moreover, $g_4(K) = \tau(K)$ implies that $\varepsilon(K) = \sgn g_4(K)$ \cite[Corollary~4]{MR3217622}.
\end{proof}

Alternating knots that are positive or negative have been studied under the name of special alternating knots (see e.g.~\cite{MR0099664}).  We may replace the sets $\mathfrak{T}_{\pm}$ in \cref{def:general_sqzd_obstruction} by the sets of alternating positive and negative knots respectively, thus defining what we might call \emph{special} squeezedness obstructions.  The proof of \cref{lem:framework} goes through entirely unchanged in this case.  Slice torus invariants and the invariant $\frac{1}{2}\sigma \colon \cC \rightarrow \Z$ are both special squeezedness obstructions hence the following result is immediate.
\begin{proposition}\label{prop:altsqueezedsti=sig}
	\label{prop:slice_torus_boring_on_alternating_sqzd}
	If $K$ is squeezed between a positive alternating knot and a negative alternating knot and $\phi$ is a slice torus invariant, then we have
	\[ \phi(K) = -\sigma(K)/2 {\rm .}  \myqed \]
\end{proposition}
Every alternating knot has the property of being squeezed between 
a positive alternating knot and a negative alternating knot, as was shown \cref{prop:alternating_knots_are_alternating_sqzd}. So, in particular \cref{prop:altsqueezedsti=sig} is a generalization of the theorems by Ozsv\'ath-{S}zab\'o \cite{osz10} that $\tau(K) = -\sigma(K)/2$ and by Rasmussen \cite{ras3} that $s(K) = -\sigma(K)$ for alternating knots $K$.

Consider the metric space consisting of the set of continuous piecewise linear functions $f\colon [0,1]\to\R$
satisfying $f(0) = 0$,
with
\[
d(f,g) = \sup_{t\in (0,1]} \frac{|f(t) - g(t)|}{t}.
\]
This contains an isometric copy of $\Z$ where $n \in \Z$ corresponds to the linear function $t\mapsto -nt$.  Then it is easily verified that the Heegaard-Floer Upsilon invariant $\Upsilon|_{[0,1]}$ is a special squeezedness obstruction (see~\cite{ossz} for details on $\Upsilon(K)$).  In other words we have the following:

\begin{proposition}\label{prop:upsalt}
	If $K$ is squeezed between a positive alternating knot and a negative alternating knot, then the Upsilon invariant is linear with slope equal to $ \sigma(K)/2$,
	i.e.\ $\Upsilon(K)(t) = \frac{1}{2}\sigma(K)\cdot t$ for all $t\in[0,1]$.  \qed %
\end{proposition}

\begin{figure}[b]
	\centering
	\includegraphics[width=2.3in,height=1.7in]{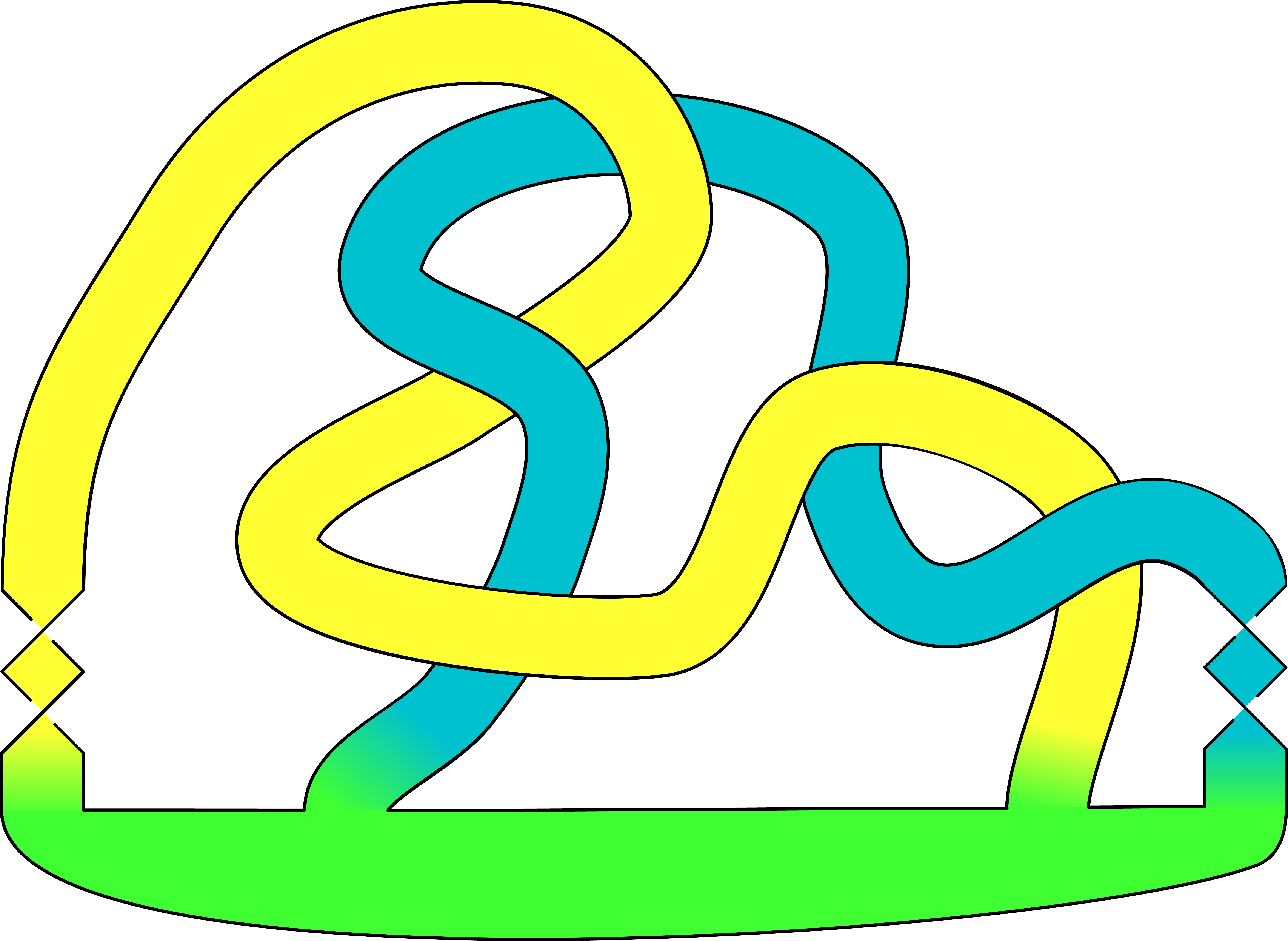}
	\caption{A strongly quasihomogeneous surface of genus one. Its boundary knot is squeezed between the positive and negative trefoil,
		but it is not quasialternating (having non-thin Khovanov homology).}
	\label{fig:17nonthin}
\end{figure}
\begin{remark}
	A class of knots that is well-known to look simple from the perspective of Heegaard-Floer knot invariants and in particular satisfies $\Upsilon(K)(t) = \frac{1}{2}\sigma(K)\cdot t$  is the
	\emph{quasialternating} knots (introduced in \cite{osz12}, we do not repeat their definition here).
	Instead we briefly note that there is not an obvious relation between being quasialternating and being squeezed.
	There are knots such as $10_{125}$, which is quasialternating \cite{greene2},
	but not squeezed (see \cref{ex:notsqzd}).
	On the other hand, one easily constructs examples of knots that are slices of genus-minimizing cobordisms between positive and negative alternating knots, but that are not quasialternating
	(see \cref{fig:17nonthin}), and in fact that are not in an obvious way concordant to a quasialternating knot.
\end{remark}

\subsection{Non-obstructions}
\label{subsec:non-obstructions}
In this subsection we consider some knot invariants that, perhaps surprisingly, do not provide squeezedness obstructions.
\begin{remark}
	The S-equivalence class and the topological concordance class of a knot
	do not provide obstructions to squeezedness. The reason is that for any knot~$K$, there is a
	strongly quasipositive (and thus squeezed) knot~$J$, such that $K$ and $J$ are S-equivalent and topologically concordant~\cite{borodzikfeller}.
\end{remark}
\begin{remark}
	The smooth concordance invariant $\varepsilon$ comes from knot Floer homology.
	The pair $(\tau, \varepsilon)$ takes values in
	\[
	\TE = \{(0,0)\} \cup (\Z \times \{-1,1\}),
	\]
	and gives a stronger lower bound for $g_4$ than $|\tau|$,
	namely $g_4(K) \geq |\tau(K)| + 1$ if $\varepsilon(K) \neq \sgn \tau(K)$ \cite[Corollary~4]{MR3217622}.
	However, the pair $(\tau, \varepsilon)$ does not obstruct squeezedness, since every value in $\TE$
	is attained by $(\tau, \varepsilon)$ on some squeezed knot. Indeed, we have $(\tau(U),\varepsilon(U)) = (0,0)$,
	and for
	\[
	K_n = T(2,2n+7) \# -\!T(3,4),
	\]
	we have for all $n\in \Z$ that
	\[
	(\tau(K_n), \varepsilon(K_n)) = (n,-1),\qquad (\tau(-K_n), \varepsilon(-K_n)) = (-n,1).
	\]
	Here, the value of $\varepsilon(K_n)$ may be computed from \cite[Lemma~6.4]{MR3260841},
	using for $n\geq -3$ that $a_1(T(2,2n+7)) = a_2(T(2,2n+7)) = a_1(T(3,4)) = 1$ and $a_2(T(3,4)) = 2$ by Lemma~6.5 in the same paper (see there for the definition of~$a_1, a_2$).
\end{remark}

In light of this remark, we ask the following.
\begin{question}
	Is there any squeezedness obstruction coming from knot Floer \mbox{homology}, other than the $\tau$ invariant (which is a slice torus invariant)?
\end{question}

Note that the obstructions to squeezedness we describe beyond the first bullet point of \cref{prop:boring_on_squeezed}, all come from (variations and the spacification of) Khovanov homology.

\begin{appendix}
\section{Squeezedness of pseudoalternating knots}
\label{app:pseudo}
It takes a little work to define pseudoalternating knots; we work up to it with a couple of intermediate definitions.

\begin{definition}
	A \emph{fountain} consists of the following data.  Firstly, a disjoint union of unnested circles in the $(x,y)$ plane.  Secondly, a totally ordered collection of signed arcs in the plane.  Each arc should be embedded, with its endpoints transverse to and lying on two different circles, and with no interior point of an arc meeting a circle.  If two arcs intersect then they do so only at interior points of the arcs.
Moreover, the graph formed by circles and arcs should be bipartite, i.e.~circles may be colored black and white such that every arc connects a black and a white circle.
\end{definition}

A fountain determines a Seifert surface in the following way.  Firstly fill each circle in the $(x,y)$ plane with the disc that it bounds.  Then for each signed arc add a half-twisted ribbon above the $(x,y)$ plane.  The sign of the arc determines whether the arc has a positive or a negative half-twist, and the total ordering on the arcs determines how the arcs cross over each other.

Suppose that $F_1$ and $F_2$ are two fountains and suppose that a circle of $F_1$ coincides with a circle of $F_2$, that any other circle of $F_1$ is disjoint from all circles of $F_2$, and that the union of all circles of $F_1$ and $F_2$ are unnested.  Furthermore suppose that the collection of all arcs of $F_1$ and $F_2$ only meet pairwise at interior points of arcs, and do not meet any circles of $F_1$ or $F_2$ at interior points.

Then we define the \emph{join} $F_1 \# F_2$ to be the fountain which has all circles and arcs of $F_1$ and of $F_2$, with the induced total ordering given by requiring that every arc of $F_2$ is above every arc of $F_1$.

\begin{definition}
	A \emph{primitive flat fountain} is a connected fountain in which the arcs all have the same sign and their interiors do not pairwise intersect.
\end{definition}

\begin{definition}
	A \emph{generalized flat fountain} is a fountain which lies in the set of fountains generated as the closure of the primitive flat fountains under the operation~$\#$.  A \emph{generalized flat surface} is the associated Seifert surface to a generalized flat fountain, and a \emph{pseudoalternating} knot is a knot that may be realized as the boundary of a generalized flat surface.
\end{definition}

This definition is slightly different, but equivalent to the original definition of pseudoalternating knots \cite{MR402718}.
With the definition in hand, the proof below is straightforward.

\begin{proof}[Proof of 	\cref{prop:all_those_classes_are_sqzd}.]
	We need to show that pseudoalternating knots are squeezed, which we shall establish by showing that they are strongly quasihomogeneous.  In fact we shall show that every generalized flat surface is a strongly quasihomogeneous surface, in other words that every generalized flat surface is the union of a strongly quasipositive and a strongly quasinegative surface where their intersection is a disc.
	
	Since we have an inductive definition of generalized flat surfaces we give an inductive proof.  Given a primitive flat fountain, consider a subfountain containing the same collection of circles, but with a minimal number of arcs so that the subfountain is still connected.  Note that the associated surface to the subfountain is a disc, which is both strongly quasipositive and strongly quasinegative.  Also note that the surface associated to any positively-signed primitive flat fountain is the canonical surface of a connected positive link diagram, and hence is a strongly quasipositive surface (there is a similar statement for negatively-signed primitive flat fountains).  Hence we see that the associated surface to any primitive flat fountain is strongly quasihomogeneous since it is the union of a strongly quasipositive and a strongly quasinegative surface which intersect in a disc; furthermore, this disc may be assumed to contain all the discs bounded by the circles of the primitive flat fountain.
	
	Suppose then for $i=1,2$, that $F_i$ is a generalized flat fountain such that $F_1 \# F_2$ makes sense.  Suppose further that the associated surface to each $F_i$ has been shown to be strongly quasihomogeneous as the union of a strongly quasipositive and a strongly quasinegative surface intersecting in a disc containing all the discs bounded by the circles of $F_i$.
	
	Then note that the surface associated to $F_1 \# F_2$ is a Murasugi sum of the surfaces associated to $F_1$ and $F_2$.  Indeed $F_1 \# F_2$  can be taken to be the union of a Murasugi sum of the strongly quasipositive subsurfaces of $F_1$ and $F_2$ with a Murasugi sum of the strongly quasinegative subsurfaces of $F_1$ and $F_2$.  The operation of Murasugi summing preserves strong quasipositivity and strong quasinegativity \cite{Rudolph_98_QuasipositivePlumbing}.  These two subsurfaces of $F_1 \# F_2$ meet in the Murasugi sum of two discs (which is a disc itself), and this sum contains all the discs bounded by the circles of $F_1 \# F_2$.
\end{proof}

\end{appendix}
\bibliographystyle{myamsalpha}
\bibliography{References}
\end{document}